\documentclass[12pt]{amsart}
 \usepackage[colorlinks=true,urlcolor=blue, citecolor=red,linkcolor=blue,linktocpage,pdfpagelabels, bookmarksnumbered,bookmarksopen]{hyperref}
 \usepackage[hyperpageref]{backref}
\usepackage{cleveref}
\usepackage{xcolor}
\usepackage{amsthm} 
\usepackage{latexsym,amsmath,amssymb}
\usepackage[colorinlistoftodos,prependcaption,textsize=tiny]{todonotes}
\usepackage{a4wide}
\usepackage{soul}

\title{H\"older continuity of Minimizing $W^{s,p}$-Harmonic Maps}

\author{Akshara Vincent}
\address[Akshara Vincent]{Department of Mathematics,
University of Pittsburgh,
301 Thackeray Hall,
Pittsburgh, PA 15260, USA}
\email{akv20@pitt.edu}
\definecolor{indigo}{rgb}{0.29, 0.0, 0.51}

\definecolor{p1}{gray}{0.4}
\definecolor{p2}{gray}{0.6}
\definecolor{p3}{gray}{0.98}
\definecolor{p4}{gray}{0.8}
\definecolor{p5}{gray}{0.9}

plus 6pt minus 12pt
\def\eps{\varepsilon}

\def\B{{B}}

\newcommand{\dif}{\,\mathrm{d}}

\def\N{{\mathbb N}}

\def\S{{\mathbb S}}

\def\R{{\mathbb R}}

\def\S{{\mathbb S}}
\def\B{{\mathbb B}}

\newtheorem{theorem}{Theorem}
\newtheorem{lemma}[theorem]{Lemma}

\newtheorem{proposition}[theorem]{Proposition}

\newtheorem{definition}[theorem]{Definition}

\newtheorem{conjecture}[theorem]{Conjecture}

\def\esssup{\mathop{\rm ess\,sup\,}}

\def\dist{{\rm dist\,}}

\newcommand{\brac}[1]{\left (#1 \right )}

\newcommand{\barint}{
\rule[.036in]{.12in}{.009in}\kern-.16in \displaystyle\int }

\newcommand{\barcal}{\mbox{$ \rule[.036in]{.11in}{.007in}\kern-.128in\int $}}

\def\mvint_#1{\mathchoice
          {\mathop{\vrule width 6pt height 3 pt depth -2.5pt
                  \kern -8pt \intop}\nolimits_{\kern -3pt #1}}%
          {\mathop{\vrule width 5pt height 3 pt depth -2.6pt
                  \kern -6pt \intop}\nolimits_{#1}}%
          {\mathop{\vrule width 5pt height 3 pt depth -2.6pt
                  \kern -6pt \intop}\nolimits_{#1}}%
          {\mathop{\vrule width 5pt height 3 pt depth -2.6pt
                  \kern -6pt \intop}\nolimits_{#1}}}

\newcommand{\RN}[1]{
  \textup{\uppercase\expandafter{\romannumeral#1}}
}

\numberwithin{theorem}{section} \numberwithin{equation}{section}

\newcommand{\lap}{\Delta }
\newcommand{\aleq}{\precsim}

\let\latexchi\chi
\makeatletter
\renewcommand\chi{\@ifnextchar_\sub@chi\latexchi}
\newcommand{\sub@chi}[2]{
  \@ifnextchar^{\subsup@chi{#2}}{\latexchi^{}_{#2}}%
}
\newcommand{\subsup@chi}[3]{
  \latexchi_{#1}^{#3}%
}
\makeatother

\makeatletter
\def\tikz@arc@opt[#1]{
  {%
    \tikzset{every arc/.try,#1}%
    \pgfkeysgetvalue{/tikz/start angle}\tikz@s
    \pgfkeysgetvalue{/tikz/end angle}\tikz@e
    \pgfkeysgetvalue{/tikz/delta angle}\tikz@d
    \ifx\tikz@s\pgfutil@empty%
      \pgfmathsetmacro\tikz@s{\tikz@e-\tikz@d}
    \else
      \ifx\tikz@e\pgfutil@empty%
        \pgfmathsetmacro\tikz@e{\tikz@s+\tikz@d}
      \fi%
    \fi
    \tikz@arc@moveto
    \xdef\pgf@marshal{\noexpand%
    \tikz@do@arc{\tikz@s}{\tikz@e}
      {\pgfkeysvalueof{/tikz/x radius}}
      {\pgfkeysvalueof{/tikz/y radius}}}%
  }%
  \pgf@marshal%
  \tikz@arcfinal%
}
\let\tikz@arc@moveto\relax
\def\tikz@arc@movetolineto#1{%
  \def\tikz@arc@moveto{\tikz@@@parse@polar{\tikz@arc@@movetolineto#1}(\tikz@s:\pgfkeysvalueof{/tikz/x radius} and \pgfkeysvalueof{/tikz/y radius})}}
\def\tikz@arc@@movetolineto#1#2{#1{\pgfpointadd{#2}{\tikz@last@position@saved}}}
\tikzset{%
  move to start/.code=\tikz@arc@movetolineto\pgfpathmoveto,%
  line to start/.code=\tikz@arc@movetolineto\pgfpathlineto}
\makeatother

\usepackage{tikz}
\usetikzlibrary{arrows}

\makeatletter
\newcommand\@erelb@r[1]{%
  \mathrel{\tikz[baseline=-.5ex]\draw[#1] (0,0)--(0.3,0);}
}

\newcommand{\erelbar}[1]{\@erelbar#1}
\def\@erelbar#1#2{%
  \ifcase\numexpr#1*4+#2\relax
    \@erelb@r{-}\or     
    \@erelb@r{->}\or    
    \@erelb@r{-|}\or    
    \@erelb@r{->|}\or   
    \@erelb@r{<-}\or    
    \@erelb@r{<->}\or   
    \@erelb@r{<-|}\or   
    \@erelb@r{<->}\or   
    \@erelb@r{|-}\or    
    \@erelb@r{|->}\or   
    \@erelb@r{|-|}\or   
    \@erelb@r{|<->|}\or 
    \@erelb@r{|<-}\or   
    \@erelb@r{|<->}\or  
    \@erelb@r{|<-|}\or  
    \@erelb@r{|<->|}    
  \else
    \@wrong
  \fi
}
\makeatother

\begin{document}
\begin{abstract}
We show that the mappings $u\in \dot{W}^{s,p}(\mathbb{R}^n,\mathcal{N})$ into manifolds $\mathcal{N}$  of a sufficiently simple topology that minimize the energy
$$\int_{\mathbb{R}^n}\int_{\mathbb{R}^n}\frac{|u(x)-u(y)|^p}{|x-y|^{n+sp}} \;dx\;dy$$
are locally H\"older continuous in a bounded domain $\Omega$ outside a singular set $\Sigma $ with Hausdorff dimension strictly smaller than $n-sp$. We avoid the use of a monotonicity formula (which is unknown if $p \neq 2$) by using a blow-up argument instead.
\end{abstract}
\maketitle
\tableofcontents

\section{Introduction}
Let $\mathcal{N}$ be a compact smooth manifold without boundary isometrically embedded into $\R^N$. Harmonic maps from a $n$-dimensional domain $\Omega \subset \R^n$ into $\mathcal{N}$ are critical points of the Dirichlet energy
\[
\int_{\Omega} |\nabla u|^2 \quad \text{s.t. $u: \Omega \to \mathcal{N}$}.
\]
Their regularity theory is quite interesting, as the following illustrates.
\begin{itemize}
    \item Harmonic maps, i.e., finite energy solutions to the Euler-Lagrange equation
    \[
\lap u = A(u)(\nabla u,\nabla u) \quad \text{in $\Omega$,}
    \]
    where $A$ is the second fundamental form of the isometric embedding $\mathcal N\subset \R^N$,
    are smooth in the critical dimension $n=2$ \cite{HeleinBook,RiviereConsLaws}. But in dimension $n \geq 3$, they can be everywhere discontinuous \cite{RivEvDisc}.
    \item \emph{Stationary} harmonic maps, i.e., finite energy solutions to the Euler-Lagrange equation that additionally are critical with respect to interior variations,
    \[
\frac{d}{d\eps } \Big |_{\eps = 0} E(u\circ \tau_\eps) = 0
    \]
    for $\tau: [-1,1] \times \Omega \to \Omega$, a smooth one parameter family of diffeomorphisms with $\tau_0 = id$, can have discontinuities in dimension $n \geq 3$. However, their singular set $\Sigma$ consisting of all points where $u$ is discontinuous has $(n-2)$-Hausdorff measure zero \cite{EvansSphere, BethuelStationary}.
    \item For $n \geq 3$, \emph{minimizing} harmonic maps, i.e., maps that satisfy
    \[
E(u) \leq E(v)
    \]
for all $v: \Omega \to \mathcal{N}$ with $v = u$ on $\partial \Omega$, can also be discontinuous. Indeed, $x/|x|$ is an example of a minimizing harmonic map from the unit ball $\B^3$ into the unit sphere $\S^2$ \cite{BCL}. The singular set $\Sigma$ of minimizing harmonic maps has Hausdorff dimension strictly less than $n-2$, $\dim_\mathcal H \Sigma \le n-3$ \cite{SU82}.
\end{itemize}
The regularity theory of harmonic maps in dimension $2$ relies on compensation phenomena and harmonic analysis techniques for the critical Euler-Lagrange equation. The arguments for stationary and minimizing maps usually rely on a monotonicity formula of the localized energy. The monotonicity formula for stationary harmonic maps is slightly weaker than the one for minimizing harmonic maps, leading to a slightly weaker regularity statement.

In this work, we are interested in the regularity of $W^{s,p}$-harmonic maps, where the Dirichlet energy is replaced with the Gagliardo-seminorm for fractional Sobolev spaces.
We have a few preliminary definitions before stating the main theorem.
\begin{definition}
    Let $1\leq p<\infty$ and $0<s<1$. Let $\Omega$ be an open subset of $\R^n$. The fractional Sobolev space, denoted by $W^{s,p}(\Omega,\R^N)$, is defined as the set of functions such that
    $$\Vert f \Vert_{W^{s,p}(\Omega,\R^N)}:= [f]_{W^{s,p}(\Omega,\R^N)}+ \Vert f \Vert_{L^p(\Omega,\R^N)}<+\infty$$
    where $[\cdot]_{W^{s,p}(\Omega,\R^N)}$ is the Gagliardo seminorm given by
    $$[f]_{W^{s,p}(\Omega,\R^N)}:= \left(\int_{\Omega}\int_{\Omega}\frac{|f(x)-f(y)|^p}{|x-y|^{n+sp}}dx\;dy\right)^{\frac{1}{p}}.$$
    $W_0^{s,p}(\Omega,\R^N)$ denotes the closure of $C^\infty_0(\Omega,\R^N)$ in the norm $\Vert \cdot \Vert_{W^{s,p}(\Omega,\R^N)}$. The notation $u\in W^{s,p}_{loc}(\R^n,\R^N)$ means $u\in  W^{s,p}(\Omega,\R^N)$ for any bounded open set $\Omega \subset \R^n$.

    The space $\dot{W}^{s,p}(\Omega,\R^N)$ is the set of functions $f\in L^{p}_{loc}(\Omega,\R^N)$ such that $[f]_{W^{s,p}(\Omega,\R^N)}<\infty$.

    Moreover, for a set $\mathcal{N} \subset \R^N$, the notion $u \in W^{s,p}(\Omega,\mathcal{N})$ means $u \in W^{s,p}(\Omega,\R^N)$ and $u \in \mathcal{N}$ a.e. in $\Omega$.
\end{definition}
\begin{definition}
    Let $2\leq p<\infty$ and $0<s<1$. Let $\Omega$ be an open subset of $\R^n$. For $f\in W^{s,p}(\Omega,\R^N)$, the fractional $p$-Laplacian of order $s$, denoted by $(\Delta_p)^s$, is defined by
    $$(-\Delta_p)^s f(x):=2\;\lim_{\varepsilon\to 0}\int_{\R^n\backslash B_\varepsilon (x)}\frac{|f(x)-f(y)|^{p-2}(f(x)-f(y))}{|x-y|^{n+sp}}dy. $$
\end{definition}
Note that the above equation is vector valued with components as
$$(-\Delta_p)^s f^i(x):=2\;\lim_{\varepsilon\to 0}\int_{\R^n\backslash B_\varepsilon (x)}\frac{|f(x)-f(y)|^{p-2}(f^i(x)-f^i(y))}{|x-y|^{n+sp}}dy $$
for $1\leq i\leq N$. The operator $(-\Delta_p)^s$ is a fractional (or nonlocal) version of the $p$-Laplacian operator, given by
$$\Delta_p f =\text{div}(|\nabla f|^{p-2}\nabla f ).$$ In weak sense, the fractional $p$-Laplacian naturally arises as a first order variation of the fractional $p$-energy functional
$$f \erelbar{21} \int_{\R^n}\int_{\R^n}\frac{|f(x)-f(y)|^p}{|x-y|^{n+sp}}dx\;dy.$$

In the distributional sense, we have $\varphi \in C_c^\infty(\R^n,\R^N)$
\[
 (-\Delta_p)^s f[\varphi]:=\int_{\R^n}\int_{\R^n}\frac{|f(x)-f(y)|^{p-2}(f(x)-f(y)) (\varphi(x)-\varphi(y))}{|x-y|^{n+sp}} dx\, dy.
\]

\begin{definition}
    Let $0<s<1$ and $1\leq p<\infty$ and $\mathcal{N}$ be a compact submanifold of $\R^N$. We say $u\in \dot{W}^{s,p}(\R^n,\mathcal{N})$ is a minimizing $W^{s,p}$-harmonic map in a bounded open set $\Omega \subset \R^n $, if for any $v\in \dot{W}^{s,p}(\R^n,\mathcal{N})$ with $v= u$ outside $\Omega$ we have
    $$\int_{\R^n}\int_{\R^n}\frac{|u(x)-u(y)|^p}{|x-y|^{n+sp}} \;dx\;dy \leq \int_{\R^n}\int_{\R^n}\frac{|v(x)-v(y)|^p}{|x-y|^{n+sp}} \;dx\;dy.$$
\end{definition}
The crucial difficulty in switching from the classical Dirichlet energy to the fractional seminorm is the lack of a monotonicity formula when $p \neq 2$. When $p=2$, it was shown in \cite{MS15} that a monotonicity formula can be obtained from the representation of the $W^{1/2,2}$-energy as the trace of the $W^{1,2}$-energy, and indeed the theory of singularities is understood to a similar extent as the classical harmonic map case \cite{MPS21}. When $p \neq 2$, it is an interesting and apparently very difficult open question how to obtain a (useful) monotonicity formula.

However, even in the absence of a workable monotonicity formula, we can estimate the size of the singular set of $W^{s,p}$-harmonic maps into manifolds in certain regimes: more precisely, the following is the main result of our paper.
\begin{theorem}\label{th:mainthm}
    Assume $1<p_1 < p_2 <\infty$ and  $0<s_1 < s_2<1$. There exists $\delta = \delta(s_1,s_2,p_1,p_2,n,N)> 0$ such that the following holds:
    Let $s \in (s_1,s_2)$ and $p \in (p_1,p_2)$ such that $n-\delta<sp<n$. Then for some $\zeta=\zeta(n,p,s)>0$ we have the following.

    Let $\Omega$ be any bounded domain in $\R^n$ with smooth boundary, and $\mathcal{N}$ be a compact smooth manifold of $\R^N$ with homotopy groups $\pi_0(\mathcal{N})=\pi_1(\mathcal{N})=...=\pi_{\lambda-2}(\mathcal{N})=0$ where $\lambda\in \mathbb{N}$ be such that $\max\{p,2\}<\lambda \leq N$. If $u\in \dot{W}^{s,p}(\R^n,\mathcal{N})$ is a minimizing $W^{s,p}$-harmonic map in $\Omega $, then $ u$ is locally H\"older continuous in $\Omega$ outside a singular set $\Sigma $ with $\mathcal{H}^{n-sp-\zeta}(\Sigma)=0$.
\end{theorem}

A few remarks are in order. Our proof relies on two steps, first an initial estimate on the singular set, which adapts ideas from \cite{HKL} in the setting of static liquid crystals. Crucially, they (and we) are able to avoid using a monotonicity formula and instead employ a hybrid inequality and a blowup argument. This leads to a singular set of dimension $n-sp$. Secondly, we use Caccioppoli's inequality to obtain a slight dimension reduction leading to the result at hand. Let us mention that in \cite{Gastel}, by using a similar technique partial regularity was obtained for manifolds with sufficiently simple topology in the setting of biharmonic maps.

The small $\delta$ in the condition $n-\delta < sp < n$ is unfortunate, and is not actually related to the difficulties of harmonic maps, but to the unknown regularity theory for homogeneous fractional $p$-Laplace \emph{systems}. Indeed, since we employ a blow-up argument, we are able to compare minimizing harmonic maps into manifolds to harmonic maps into Euclidean space. But, sadly, the following seemingly elementary result is not known:

\begin{conjecture}\label{th:BLS paper thm 1.4}
Let $n,N \in \N$ and $\Omega$ be a bounded open set in $\R^n$. Let $u \in W^{s,p}_{loc}(\Omega,\R^N)$ for $s \in (0,1)$ and $p \in (1,\infty)$ be a solution to
\[
\int_{\R^n} \int_{\R^n} \frac{|u(x)-u(y)|^{p-2}(u(x)-u(y)) (\varphi(x)-\varphi(y))}{|x-y|^{n+sp}}\, dx\, dy = 0 \quad \forall \varphi \in C_c^\infty(\Omega,\R^N),
\]
Then $u$ is locally H\"older continuous.
\end{conjecture}
The above conjecture is true for $N=1$, see, e.g. \cite{DCKP16,BrascoLingrenSchikorra} and references therein. For $N \geq 2$ this question is wide open; indeed, only the boundedness of solutions is known, and this was obtained very recently \cite{BDNS24}.

If \Cref{th:BLS paper thm 1.4} holds true, our $\delta$ in \Cref{th:mainthm} immediately improves without many changes. But,  since we cannot prove the above conjecture, in our case, we have to rely on the fact that there exists a $\delta > 0$ such that \Cref{th:BLS paper thm 1.4} holds whenever $n-\delta<sp \leq n$. More precisely,

\begin{theorem}\label{th:alternate for BLS}
Let $\Omega\subset \R^n$ be a bounded and open set. Assume $1<p_1 < p_2 <\infty$ and  $0<s_1 < s_2<1$. There exists $\delta = \delta(s_1,s_2,p_1,p_2,n,N)> 0$ such that the following holds for any $s \in (s_1,s_2)$ and $p \in (p_1,p_2)$ such that $sp+\delta>n$:

If $u\in W^{s,p}_{\text{loc}}(\Omega,\R^N)\cap L^{p-1}_{sp}(\R^n,\R^N)$ is a local weak solution of
$$(-\Delta_p)^su=f\;\;\;\;\;\text{in}\;\Omega,$$
where $f\in L^{\frac{p}{p-1}}_{\text{loc}}(\Omega,\R^N)$, then $u\in C^{0,\alpha}_{\text{loc}}(\Omega,\R^N)$ where $\alpha:=(sp+\delta-n)/p$.

 Moreover, we have for any ball $B_{4R}(x_0)\subset \subset \Omega$, there exists a uniform constant $C=C(n,s,p,\delta)>0$ such that
\begin{equation}\label{th:1.4}
\begin{split}
    [u]&_{C^{0,\alpha}(B_{R}(x_0),\R^N)} \\&\leq CR^{-\frac{\delta}{p}} \brac{\int_{B_{4R}(x_0)} \int_{\R^n} \frac{|u(x)-u(y)|^{p}}{|x-y|^{n+sp}}\, dx dy  + \|f\|_{L^{\frac{p}{p-1}}(B_{2R}(x_0),\R^N)} [u]_{W^{s,p}(B_{2R}(x_0),\R^N)}}^{\frac{1}{p}}.
\end{split}
\end{equation}

\end{theorem}

The proof of \Cref{th:alternate for BLS} is quite standard and we sketch it in \Cref{s:nongeomplapsys}.
\subsection{Outline of the paper} In \Cref{section prelim}, we introduce the preliminary notations and theorems used in the paper. \Cref{section caccio} shows that the minimizers satisfy a Caccioppoli-type estimate. The proof of this involves the $W^{s,p}$-extenstion property into compact submanifolds of $\R^N$. The proof of the extension theorem is presented in \Cref{section extension}. The significant step involved in the proof of the main theorem is the observation that the minimizers follow a decay estimate which is given in \Cref{section decay}. Precisely, the normalized energy
$$R^{sp-n}\int_{B_R(x_0)}\int_{\R^n}\frac{|u(x)-u(y)|^p}{|x-y|^{n+sp}}dx\;dy$$
decays as $R\to 0$. The proof of the decay estimate follows by the method of contradiction. We look at a sequence of minimizers which do not exhibit the decay estimate but which have total energy
$$\int_{B_1}\int_{\R^n}\frac{|u(x)-u(y)|^p}{|x-y|^{n+sp}}dx\;dy=\varepsilon_i^p$$
which converges to $0$ as $i\to \infty$. We define a sequence of normalized functions $v_i=\varepsilon_i^{-1}(u_i-(u_i)_{B_1})$ and observe that these functions converge to a $W^{s,p}$-limit $v$ which satisfies
$$(-\Delta_p)^s v=g\;\;\;\;\text{in}\;\;\;B_1.$$
Then, we use the regularity of $v$ and the Caccioppoli-type estimate for the minimizers to get a contradiction. In \Cref{section proof}, we prove the main theorem. The singular set $\Sigma$ is defined as the set of all points $x_0$ around which the normalized energy fails to approach zero as $R\to 0$. We use Campanato's theorem to show that the minimizers are locally H\"older continuous around points with a decay estimate. The idea of this proof is motivated by results from Hardt and Lin \cite{HardtLin87} where they proved the H\"older continuity of mappings between manifolds minimizing the energy
$$\int_M |\nabla u|^p.$$
By using Caccioppoli-type estimate, we reduced the dimension of the singular set $\Sigma$ to a dimension strictly smaller than $n-sp$.
\subsection{Assumptions.}\label{assumptions} Throughout the paper, we assume that $\Omega$ is a bounded domain in $\R^n$. For simplicity, we assume that $\Omega$ has a smooth boundary. For given $1< p <\infty$, we assume that $\mathcal{N}$ is a compact smooth $(\lambda-2)$- simply connected manifold of $\R^N$, i.e., the homotopy groups $\pi_0(\mathcal{N})=\pi_1(\mathcal{N})=...=\pi_{\lambda-2}(\mathcal{N})=0$ where $\lambda\in \mathbb{N}$ is such that $\max\{p,2\}<\lambda \leq N$.
\subsection*{Acknowledgement}
A.V. received funding through NSF Career DMS-2044898.  The author would like to thank Katarzyna Mazowiecka and Armin Schikorra for the helpful discussions.

\section{Preliminaries and Notations}\label{section prelim}
\subsection{Notations} For $x=(x^1,...,x^n)\in \R^n$, $|\cdot|$ denotes the standard Euclidean norm in $\R^n$ given by
$|x|=\sqrt{(x^1)^2+...+(x^n)^2}.$

For $r>0$ and $x_0\in \R^n$, the ball centered at $x_0$ of radius $r$ is given by
$$B_r(x_0)=\{x\in \R^n:\;|x-x_0|<r\}.$$
If $x_0=0$, we denote $B_r(0)$ as $B_r$ for simplicity.

Let $\Omega$ be an open subset of $\R^n$ and $f: \Omega \to \R^N$ be a measurable function. $\mathcal{L}^n(\Omega)$ denotes the $n$-dimensional Lebesgue measure of $\Omega$ and the mean value is given by
$$(f)_{\Omega}=\frac{1}{\mathcal{L}^n(\Omega)}\int_{\Omega}f=\barint_{\Omega}f.$$

For $\alpha>0$, $\mathcal{H}^\alpha(\Omega)$ denotes the $\alpha$-dimensional Hausdorff measure of the measurable subset $\Omega$ of $\R^n$.

For real numbers $a$ and $b$, the expression $a \aleq b$ means $a\leq C b$ for some constant $C>0$ that depends on the parameters, domain or range, but not on the functions involved in the result. Similarly, $a\approx b$ means $a \aleq b$ and $b \aleq a$. We use $\aleq_{\gamma}$ to  indicate that the constant depends on $\gamma$.

\begin{definition}
    Let $0<\alpha\leq 1$. Let $\Omega$ be an open subset of $\R^n$. The H\"older space, denoted by $C^{0,\alpha}(\Omega,\R^N)$, is defined as the set of functions such that
    $$\Vert f \Vert_{C^{0,\alpha}(\Omega,\R^N)}:= [f]_{C^{0,\alpha}(\Omega,\R^N)}+ \Vert f \Vert_{L^\infty(\Omega,\R^N)}<+\infty,$$
    where $[\cdot]_{C^{0,\alpha}(\Omega,\R^N)}$ is the seminorm given by
    $$[f]_{C^{0,\alpha}(\Omega,\R^N)}:= \sup_{\substack{x,y\in \Omega\\x\neq y}}\frac{|f(x)-f(y)|}{|x-y|^\alpha}.$$
\end{definition}
\begin{definition}\label{def:tailspace}
    Let $p>0$ and $\alpha>0$. We introduce the \textit{tail space}
    $$L^p_\alpha(\R^n,\R^N):=\left\{f\in L^p_{loc}(\R^n,\R^N):\;\int_{\R^n}\frac{|f(x)|^p}{1+|x|^{n+\alpha}}dx<+\infty\right\}.$$
\end{definition}
We have the following theorems and lemmas which will be used in proving the main result of the paper. The first one is called Campanato's theorem which gives a characterization of H\"older continuous functions, see, e.g., \cite[Theorem 5.5]{giaquinta-martinazzi}.
\begin{theorem}[Campanato's theorem]\label{th:campanato}
    Let $\Omega$ be a domain in $\R^n$ with smooth boundary. Let $f\in L^p(\Omega,\R^N)$ and assume that for $n<\lambda\leq n+p$
      $$\Lambda:=\sup_{B_r(x)\subset \Omega}r^{-\lambda} \int_{B_r(x)}|f-(f)_{B_r(x)}|^p<\infty,$$
      then $f\in C^{0,\frac{\lambda-n}{p}}(\overline{\Omega},\R^N)$.
\end{theorem}
A more general statement of the following lemma and its proof can be found in \cite[Lemma 3.2.2]{Ziemer}.
\begin{lemma}\label{Frostmann}
    Let $p\in \left[1,\infty\right)$ and $\alpha\in \left[0,n\right).$ Assume that $f\in W^{s,p}(\R^n,\R^N)$ and define the set
$$E:=\left\{x:\;\;\limsup_{r\to 0}r^{-\alpha}\int_{B_r(x)}\int_{\R^n}\frac{|f(x)-f(y)|^p}{|x-y|^{n+sp}}dx\; dy>0\right\}.$$
Then $\mathcal{H}^\alpha(E)=0.$
\end{lemma}
Next, we have the following embedding theorem \cite[Theorem 7.1]{Hitchhikers}.
\begin{theorem}[Rellich-Kondrachov theorem]\label{th:Rellish}
Let $\Omega \subset \R^n$ be an open set with bounded Lipschitz continuous boundary. Let $1\leq p<\infty$ and $0<s<1$. Assume that $\{f_k\}_{k\in \N}\subset  W^{s,p}(\Omega)$ is bounded, i.e.,
$$\sup_{k\in \N}\Vert f_k\Vert_{W^{s,p}(\Omega)}<\infty.$$
Then there exists a subsequence $\{f_{k_i}\}\subset\{f_k\}_{k\in \N} $ and $f\in L^p(\Omega)$ such that $f_{k_i}$ converges to $f$ strongly in $L^p(\Omega)$, and moreover, the convergence is pointwise a.e.
\end{theorem}
The proof of the following density theorem can be found in \cite[Theorem 6.62]{giovanni}.
\begin{theorem}\label{th:density giovanni}
   Let $1\leq p<\infty$ and $0<s<1$. For every $f\in \dot{W}^{s,p}(\R^n)$ there exists a sequence $\{f_k\}_k$ in $\dot{W}^{s,p}(\R^n)\cap C^\infty(\R^n)$ such that
   $$\Vert f_k\Vert_{L^p(\R^n)}\leq \Vert f\Vert_{L^p(\R^n)},\;\;[f_k]_{W^{s,p}(\R^n)}\leq [ f]_{W^{s,p}(\R^n)},$$
   and
   $$[ f -f_k]_{W^{s,p}(\R^n)}\to 0$$
   as $k\to \infty$. In particular, $W^{s,p}(\R^n)\cap C^\infty(\R^n)$ is dense in $W^{s,p}(\R^n)$.
\end{theorem}
We have a local higher integrability result for non-negative functions, called Gehring's lemma. For $\sigma =2$, the proof can be found in \cite[Theorem 1.5]{gehring}. It is easy to adapt their proof to a general $\sigma > 1$.
\begin{theorem}\label{th:gehring}
    Let $\Omega$ be a domain in $\R^n$, $1<p<\infty$, and $\sigma >0$. Assume $f\in L^{p}_{loc}(\Omega)$ is a non-negative function which satisfies
    $$\left(\barint_Q f^p\right)^{1/p}\leq C_1 \barint_{\sigma Q} f$$
    for each cube $Q$ such that $\sigma Q\subset \Omega$ with the constant $C_1\geq 1$ independent of the cube $Q$. Then there exist $q>p$ so that
    $$\left(\barint_Q f^q\right)^{1/q}\leq C_2 \left(\barint_{\sigma Q} f^p\right)^{1/p},$$
    where the constant $C_2\geq 1$ is independent of the cube $Q$. In particular, $f\in L^{q}_{loc}(\Omega)$.
\end{theorem}
We conclude this section with the following lemma \cite[Lemma C.4]{kasha-schikorra-functional}.
\begin{lemma}\label{Lemma C.4}
    Let $s\in (0,1)$ and $q,p_2\in (1,\infty)$. Let $f\in C^{\infty}_c(\R^n)$. For any $p\in (1,p_2)$, if
    $$s-\frac{n}{q}=-\frac{n}{p_2}$$
    then there is some $\Lambda>1$ so that for any ball $B_r(x_0)$ we have
    $$\Vert f-(f)_{B_r(x_0)}\Vert_{L^{p_2}(B_r(x_0))} \leq C\left(\int_{B_{\Lambda r}(x_0 )}\left(\int_{B_{\Lambda r}(x_0)}\frac{|f(x)-f(y)|^p}{|x-y|^{n+sp}}dx\right)^{\frac{q}{p}}dy\right)^{\frac{1}{q}}.$$
    The constant $C$ depends only on $s,p,p_2,q$ and the dimension.
\end{lemma}
\subsection{Extension theorem} \label{section extension}
An underlying tool used in the proof of the \Cref{th:caccioppoli} is the $W^{s,p}$-extension into manifolds $\mathcal{N}$ with sufficiently simple topology. The general $W^{1,p}$-extension property of $W^{1-\frac{1}{p},p}$-maps had been proved by Hardt and Lin \cite[Theorem 6.2]{HardtLin87} under certain topological assumptions on the manifold $\mathcal{N}$. 
Inspired by their results, we provide a proof for the extension theorem in the class of $W^{s,p}$-maps.

An important step involved in proving the extension theorem is the following lemma \cite[Lemma 6.1]{HardtLin87}, which was several times reproved \cite{schaftgen_2014,Hopper,MMS}.
\begin{lemma}\label{lemma 6.1 lifting}
    Assume that $\mathcal{N}\subset \R^N$ is a compact smooth manifold of $\R^N$ with $\pi_0(\mathcal{N})=\pi_1(\mathcal{N})=...=\pi_{\lambda-2}(\mathcal{N})=0$, contained in a large cube $[-R,R]^N$. Then, there exists a compact $(N-\lambda)$-dimensional Lipschitz polyhedron $Y\subset [-2R,2R]^N$ and a locally Lipschitz retraction $P:[-2R,2R]^N\backslash Y\to \mathcal{N}$ such that
    \begin{itemize}
        \item[(1)] for some small $\rho>0$, the restriction $P|_{B_\rho(\mathcal{N})}$ is the nearest point projection to $\mathcal{N}$;
        \item[(2)] $|\nabla P(x)| \leq C \dist(x,Y)^{-1}$ for $x\notin Y$;
        \item[(3)] $\int_{[-2R,2R]^{N}}|\nabla P(x)|^pdx<\infty$ whenever $1<p<\lambda$.
    \end{itemize}
\end{lemma}
In the following lemma, the first case with $\lambda=1$ is a special case in \cite[Lemma 2.3]{schaftgen_2014} where the proof has been presented. For $\lambda>1$, we can take $Y$ as a subset of a finite union of $N-1$ dimensional planes and apply the first case to prove the desired result.  
\begin{lemma}\label{le: well defineness of Pcirc v}
    Let $\Omega\subset \R^n$ be an open bounded set, $u\in C^\infty(\overline{\Omega},\R^N)$ and let $\lambda\in \{1,\ldots,N\}$. If $Y\subset \R^N$ is a finite union of $N-\lambda$ dimensional planes, then for $\mathcal{L}^N$-almost every $a\in \R^N$, $u^{-1}(Y+a)$ is a set of $\mathcal{L}^n$-measure zero in $\Omega$.
\end{lemma}
\begin{theorem}[Extension theorem]\label{tm:extension theorem}
    Let $1<p<\infty$ and $0<s<1$. Let $\Omega\subset \R^n$ be a bounded domain with smooth boundary and $\mathcal{N}\subset \R^N$ be a compact smooth manifold of $\R^N$ with $\pi_0(\mathcal{N})=\pi_1(\mathcal{N})=...=\pi_{\lambda-2}(\mathcal{N})=0$ where $\lambda\in \mathbb{N}$ be such that $\max\{p,2\}<\lambda \leq N$. Assume that we have a map $v\in \dot{W}^{s,p}(\R^n,\R^N)$ with $v(x)\in \mathcal{N}$ for almost every $x\in \R^n\backslash \Omega$. Then there exists a map $u\in \dot{W}^{s,p}(\R^n, \mathcal{N})$ such that $u(x)=v(x)$ for almost every $x\in \R^n\backslash \Omega$ and
    $$\int_\Omega\int_{\R^n}\frac{|u(x)-u(y)|^p}{|x-y|^{n+sp}}\;dx\;dy\;\leq C\int_\Omega\int_{\R^n}\frac{|v(x)-v(y)|^p}{|x-y|^{n+sp}}\;dx\;dy$$
    for a uniform constant $C$ depending only on $p$ and $\mathcal{N}$.
\end{theorem}
In particular, the theorem above applies to the sphere $\S^{N-1} \subset \R^N$ (with $\lambda = N$) and any set diffeomorphic to it.
\begin{proof}[Proof of \Cref{tm:extension theorem}]
    Choose $R>0$ such that $\mathcal{N}\subset [-R,R]^N$. Let $P:[-2R,2R]^N\backslash Y\to \mathcal{N}$ and $\rho$ be as in \Cref{lemma 6.1 lifting}. 
    Take any arbitrary open bounded set $\Omega_0$ with smooth boundary such that $\overline{\Omega}\subset \Omega_0\subset \R^n$. By \Cref{th:density giovanni}, there exists a sequence $\{v_k\}_k\subset \dot{W}^{s,p}(\R^n,\R^N)\cap C^\infty(\R^n,\R^N )$
    such that $\Vert v_k\Vert_{L^p(\R^n,\R^N)}\leq\Vert v\Vert_{L^p(\R^n,\R^N)}$ and $[v_k- v]_{W^{s,p}(\R^n,\R^N)}\to 0$ as $k\to \infty$. Without loss of generality, we can choose the sequence such that
    \begin{equation*}
        \int_{\Omega_0}\int_{\R^n}\frac{|v_k(x)-v_k(y)|^p}{|x-y|^{n+sp}}dx\;dy\leq \int_{\Omega_0}\int_{\R^n}\frac{|v(x)-v(y)|^p}{|x-y|^{n+sp}}dx\;dy.
    \end{equation*}
    First, assume that $v$ and $v_k$ take values in $[-R,R]^N$.

    Let $\mathcal{P}:B_\rho(\mathcal{N})\to \mathcal{N}$ be the nearest point projection to the manifold $\mathcal{N}$. The restriction $\mathcal{P}|_{\mathcal{N}}:\mathcal{N}\to \mathcal{N}$ is the identity map and $\mathcal{N}$ is compact. Therefore, by the inverse function theorem we can choose $0<r<\rho$ such that
    $$\mathcal{N}\ni x \to \mathcal{P}(x-a)\in \mathcal{N}$$
    are uniformly bi-Lipschitz for all $a\in B_r$. For $a\in B_r$, we define $P_a(x):=P(x-a)$. Since $P(x)=\mathcal{P}(x)$ for $x \in B_\rho(\mathcal{N})$, there exists $\Lambda>0$ depending only on $\mathcal{N}$ such that
    \begin{equation}\label{eq:lipschitzbound for P_a inverse}
        \Lambda=\esssup_{a\in B_r(0)}\text{Lip}(P_a|_{\mathcal{N}})^{-1}<\infty.
    \end{equation}
Note that for $x,y\in [-R,R]^N$ we have
    $$|P_{a}(x)-P_{a}(y)|\leq \int_0^1|\nabla P_{a}(tx+(1-t)y)||x-y|\;dt.$$
    For each $k\in \N$, consider the composition $P_{a}\circ v_k$. Since $v_k$ is smooth on $\overline{\Omega_0}$, 
    by \Cref{le: well defineness of Pcirc v}, we get that $v_k^{-1}(Y+a)$ is a set of measure zero in $\Omega_0$ for almost every $a\in B_r(0)$. Hence, $P_a\circ v_k$ is well-defined on $\R^n$, and we have $P_{a}\circ v_k\in W^{s,p}(\Omega_0,\mathcal{N})$ for almost every $a\in B_r(0)$. Indeed,
    \begin{equation}\label{eq:w_sp bound of P_a of v}
        \begin{split}
        \int_{B_r(0)}&\int_{\Omega_0}\int_{\R^n}\frac{|(P_{a}\circ v_k)(x)-(P_{a}\circ v_k)(y)|^p}{|x-y|^{n+sp}}dx\;dy\;da\\&\leq\int_{\Omega_0}\int_{\R^n}\frac{|v_k(x)-v_k(y)|^p}{|x-y|^{n+sp}}\int_{B_r(0)}\int_0^1|\nabla P_{a}(tv_k(x)+(1-t)v_k(y))|^p\;dt\; da\;dx\;dy\\
        &\leq \int_{\Omega_0}\int_{\R^n}\frac{|v_k(x)-v_k(y)|^p}{|x-y|^{n+sp}}\int_0^1\int_{B_r(0)}|\nabla P(tv_k(x)+(1-t)v_k(y)-a)|^p\; da\;dt\;dx\;dy\\
        &\leq \int_{\Omega_0}\int_{\R^n}\frac{|v_k(x)-v_k(y)|^p}{|x-y|^{n+sp}}\int_0^1\int_{[-2R,2R]^N}|\nabla P(a)|^p\; da\;dt\;dx\;dy\\
        &\leq C \int_{\Omega_0}\int_{\R^n}\frac{|v(x)-v(y)|^p}{|x-y|^{n+sp}}dx\;dy<\infty
    \end{split}
    \end{equation}
    with $C>0$ depending only on $\mathcal{N}$. 
For $k\in \N$, we define
$$\Gamma_{1,k}:= \left\{a\in B_r(0):\text{Equation }(\ref{eq:bound for P_a comp v_k for omega 0})\text{ holds true} \right\},$$
where
\begin{equation}\label{eq:bound for P_a comp v_k for omega 0}
        \int_{\Omega_0}\int_{\R^n}\frac{|(P_{a}\circ v_k)(x)-(P_{a}\circ v_k)(y)|^p}{|x-y|^{n+sp}}dx\;dy\leq C_1\mathcal{L}^N(B_r(0))^{-1}\int_{\Omega_0}\int_{\R^n}\frac{|v(x)-v(y)|^p}{|x-y|^{n+sp}}dx\;dy
\end{equation}
for a constant $C_1>0$ chosen such that $\mathcal{L}^N(\Gamma_{1,k})>\frac{3}{4} \mathcal{L}^N(B_r(0))$ for all $k \in \N$.
\\ 
We prove this claim by method of contradiction. Assume the claim is not true. Then $\mathcal{L}^N(\Gamma_{1,k}^c)>\frac{1}{4} \mathcal{L}^N(B_r(0))$. For all  $a\in \Gamma_{1,k}^c$, we have
$$C_1\mathcal{L}^N(B_r(0))^{-1}\int_{\Omega_0}\int_{\R^n}\frac{|v(x)-v(y)|^p}{|x-y|^{n+sp}}dx\;dy<\int_{\Omega_0}\int_{\R^n}\frac{|(P_{a}\circ v_k)(x)-(P_{a}\circ v_k)(y)|^p}{|x-y|^{n+sp}}dx\;dy.$$
Integrating the above equation with respect to $a\in \Gamma_{1,k}^c$ gives
\begin{equation*}
    \begin{split}
        \frac{C_1} {4}\int_{\Omega_0}\int_{\R^n}\frac{|v(x)-v(y)|^p}{|x-y|^{n+sp}}dx\;dy\;&<\int_{\Gamma_{1,k}^c}\int_{\Omega_0}\int_{\R^n}\frac{|(P_{a}\circ v_k)(x)-(P_{a}\circ v_k)(y)|^p}{|x-y|^{n+sp}}dx\;dy\;da\\
        &\leq C\int_{\Omega_0}\int_{\R^n}\frac{|v(x)-v(y)|^p}{|x-y|^{n+sp}}dx\;dy
    \end{split}
\end{equation*}
where the last inequality follows from (\ref{eq:w_sp bound of P_a of v}). This is a contradiction if we choose $C_1$ large enough (in particular, choose $C_1=4C$).

We now define
$$\Gamma_{2,k}:= \left\{a\in B_r(0):\text{Equation }(\ref{eq:bound for P_a comp v_k for omega})\text{ holds true} \right\},$$
where
\begin{equation}\label{eq:bound for P_a comp v_k for omega}
        \int_{\Omega}\int_{\R^n}\frac{|(P_{a}\circ v_k)(x)-(P_{a}\circ v_k)(y)|^p}{|x-y|^{n+sp}}dx\;dy\leq C_1\mathcal{L}^N(B_r(0))^{-1}\int_{\Omega}\int_{\R^n}\frac{|v(x)-v(y)|^p}{|x-y|^{n+sp}}dx\;dy
\end{equation}
for some uniform constant $C_1>0$. Again, choosing $C_1$ large enough, we can ensure that $\mathcal{L}^N(\Gamma_{2,k})>\frac{3}{4} \mathcal{L}^N(B_r(0))$ for all $k \in \N$. This implies $\mathcal{L}^N(\Gamma_{1,k}\cap \Gamma_{2,k})>0$ for all $k \in \N$. Therefore, for each $k \in \N$ we can choose $a_k\in B_r(0)$ such that $(\ref{eq:bound for P_a comp v_k for omega 0})$ and $(\ref{eq:bound for P_a comp v_k for omega})$ hold simultaneously.

Note that $\{P_{a_k}\circ v_k\}_k$ is uniformly bounded in $W^{s,p}(\Omega_0,\mathcal{N})$. By weak compactness, we get, on a subsequence (denoted the same), $P_{a_k}\circ v_k\to g$ weakly in $W^{s,p}(\Omega_0,\R^N)$. By \Cref{th:Rellish}, $P_{a_k}\circ v_k\to g$ strongly in $L^{p}(\Omega_0,\R^N)$ and therefore, $P_{a_k}\circ v_k\to g$ pointwise almost everywhere in $\Omega_0$. Clearly, $g\in W^{s,p}(\Omega_0,\mathcal{N})$.

Re-labeling up to a subsequence, we may assume $a_k\to a_0 \in \overline{B_r(0)}$. Let $x\in \Omega_0\backslash\Omega$ such that $v(x)\in \mathcal{N}$ and $v_k(x)\to v(x)$. Since $|a_0|\leq r$, we get $v(x)-a_0\in B_r (\mathcal{N})$. Since $P$ is locally Lipschitz on $B_r(\mathcal{N})$, there exists $r_1>0$ such that
$$|P(z)-P(v(x)-a_0)|\leq L |z-(v(x)-a_0)|$$
for some constant $L>0$ and for all $z\in B_{r_1}(v(x)-a_0)$. Since $v_k(x)\to v(x)$ and $a_k\to a_0$, we can choose $k_0>0$ such that $v_k(x)-a_k\in B_{r_1}(v(x)-a_0)$ for all $k\geq k_0$. Therefore, for $k\geq k_0$
\begin{align*}
    |(P_{a_k}\circ v_k)(x)-(P_{a_0}\circ v)(x)|&=|P(v_k(x)-a_k)-P(v(x)-a_0)|\\
    &\leq L |(v_k(x)-a_k)-(v(x)-a_0)|\\
    &\leq L( |v_k(x)-v(x)|+|a_k-a_0|).
\end{align*}
Letting $k\to \infty$, we get $(P_{a_k}\circ v_k)(x)\to (P_{a_0}\circ v)(x)$. This happens for almost all $x\in \Omega_0\backslash \Omega$. Hence, $g=P_{a_0}\circ v$ almost everywhere on $\Omega_0\backslash \Omega$. In particular, $g(x) \in \mathcal{N}$ for almost every $x \in \Omega_0 \setminus \Omega$.
    We define $$u:= (P_{a_0}|_\mathcal{N})^{-1}\circ g.$$
    Since $v(x) \in \mathcal{N}$ for almost every $x\in \Omega_0 \setminus \Omega$, we have $u(x)=v(x)$ for almost every $x\in \Omega_0\backslash \Omega$. We can extend $u$ to $\R^n$ by defining $u:=v$ outside $\Omega_0$. Clearly, $u\in \dot{W}^{s,p}(\R^n,\mathcal{N})$.
    Letting $k\to \infty$ in (\ref{eq:bound for P_a comp v_k for omega}) and, using (\ref{eq:lipschitzbound for P_a inverse}) we get
    \begin{align*}
        \int_{\Omega}\int_{\R^n}\frac{|u(x)-u(y)|^p}{|x-y|^{n+sp}}dx\;dy&\leq \Lambda^p \int_{\Omega}\int_{\R^n}\frac{|g(x)-g(y)|^p}{|x-y|^{n+sp}}dx\;dy\\
        &\leq C' \int_{\Omega}\int_{\R^n}\frac{|v(x)-v(y)|^p}{|x-y|^{n+sp}}dx\;dy
    \end{align*}
    where $C'>0$ is a uniform constant depending only on $\mathcal{N}$ and $p$.

    The previous argument assumed that $v$ and $v_k$ take values in $[-R,R]^N$. If that is not the case, we can define $P_R:\R^N\to [-R,R]^N$ by the identity on $[-R,R]^N$ and by radial projection outside the cube. Since $P_R$ is Lipschitz with constant 1, we apply previous argument for $w:=P_R\circ v$ and for the sequence $w_k:= P_R\circ v_k$, and we get the desired result.
\end{proof}

%
\section{Caccioppoli-type Estimate}\label{section caccio}
We start by observing that minimizing harmonic maps satisfy a Caccioppoli-type estimate. We will need the following iteration lemma, which can be proved as in Lemma 3.1 from \cite[Chapter V]{Giaquinta}.

\begin{lemma}[Iteration lemma]\label{la:iteration}
 Let $0 \le a < b < \infty$ and $h \colon [a,b] \to [0,\infty)$ be a bounded function. Suppose that there are constants $\theta \in (0,1)$, $A,\,B,\,\alpha,\,\beta > 0$ such that
\begin{equation}\label{eq:iterationinequality}
 h(r) \le \theta h(R) + \frac{A}{(R-r)^\alpha}+ \frac{B}{(R-r)^\beta}
 \quad \text{for all } a \le r < R \le b.
\end{equation}
Then we obtain the bound
\[
 h(r)\le C \brac{\frac{A}{(b-r)^\alpha} + \frac{B}{(b-r)^\beta}}
 \quad \text{for all } a \le r < b
\]
with some constant $C=C(\theta,\alpha,\beta) > 0$.
\end{lemma}

\begin{theorem}[Caccioppoli type estimate]\label{th:caccioppoli}
Let $0<s<1$ and $1< p <\infty$. Assume that $\mathcal{N}$ is as in \Cref{assumptions}. Let $sp<n$ and $u\in \dot{W}^{s,p}(\R^n,\mathcal{N})$ be a minimizing $W^{s,p}$-harmonic map in $B_1$. Then, for any $\rho>0$ such that $B_{6\rho}\subset B_1$ we have
\[
\int_{B_\rho}\int_{\R^n}\frac{|u(x) - u(y)|^p}{|x-y|^{n+sp}}\dif x \dif y \le C\rho^{-s p}\int_{B_{6\rho}}|u-(u)_{B_{6\rho}}|^p
\]
where $C=C(n,s,p,\mathcal{N})>0$ is a uniform constant.
\end{theorem}

\begin{proof}
Take any $\rho>0$ such that $B_{6\rho}\subset B_1$. Let $\rho\le r< R \le 2\rho$ and let $\eta\in C^\infty_c(B_R,[0,1])$ be a cut-off function such that $\eta\equiv 1$ on $B_r$ and $\Vert\nabla \eta\Vert_{L^\infty(\R^n)}\aleq \frac{1}{R-r}$.

We define a map $v:\R^n\to \R^N$ by $v(x) := \eta(x)(u)_{B_{6\rho}} + (1-\eta(x))u(x)$. Note that $v\in \dot{W}^{s,p}(\R^n, \R^N) $. We have $v=u$ outside of $B_R$ (and in particular, on $\R^n\backslash B_1$). Thus, by \Cref{tm:extension theorem}, there exists a $w\in \dot{W}^{s,p}(\R^n, \mathcal{N})$ such that $w=v$ outside of $B_R$ and
\begin{equation}\label{eq:firstcomparison}
 \int_{B_R}\int_{\R^n} \frac{|w(x) - w(y)|^p}{|x-y|^{n+sp}}\dif x \dif y\;\leq C_1 \int_{B_R}\int_{\R^n}\frac{|v(x) - v(y)|^p}{|x-y|^{n+sp}}\dif x \dif y
\end{equation}
where $C_1>0$ is a constant independent of $u$ and $B_R$.
 We decompose
 \[\R^n\times\R^n = \brac{(\R^n\setminus B_R)\times (\R^n\setminus B_R)} \cup \brac{(\R^n\setminus B_R)\times B_R} \cup \brac{B_R\times (\R^n\setminus B_R)} \cup \brac{B_R\times B_R}.
 \]
 By minimality of $u$ and the fact $w=u$ outside of $B_R$, we get
 \begin{equation}\label{eq:second comparison}
  \begin{split}
   \int_{B_R}\int_{B_R}& \frac{|u(x) - u(y)|^p}{|x-y|^{n+sp}}\dif x \dif y+ 2\int_{B_R}\int_{\R^n\backslash B_R} \frac{|u(x) - u(y)|^p}{|x-y|^{n+sp}}\dif x \dif y\\
   &\leq \int_{B_R}\int_{B_R} \frac{|w(x) - w(y)|^p}{|x-y|^{n+sp}}\dif x \dif y+ 2\int_{B_R}\int_{\R^n\backslash B_R} \frac{|w(x) - w(y)|^p}{|x-y|^{n+sp}}\dif x \dif y.
  \end{split}
 \end{equation}
 Using \eqref{eq:firstcomparison} and \eqref{eq:second comparison}, we find
 \begin{equation}\label{eq:comparison-simplified}
  \begin{split}
   \int_{B_r}\int_{\R^n}& \frac{|u(x) - u(y)|^p}{|x-y|^{n+sp}}\dif x \dif y\\
   &\leq  \int_{B_R}\int_{\R^n\setminus B_R}\frac{|u(x) - u(y)|^p}{|x-y|^{n+sp}}\dif x \dif y + \int_{B_R}\int_{B_R} \frac{|u(x) - u(y)|^p}{|x-y|^{n+sp}}\dif x \dif y \\
   &\leq  2\int_{B_R}\int_{\R^n\setminus B_R}\frac{|u(x) - u(y)|^p}{|x-y|^{n+sp}}\dif x \dif y + \int_{B_R}\int_{B_R} \frac{|u(x) - u(y)|^p}{|x-y|^{n+sp}}\dif x \dif y \\
   &\leq  2\int_{B_R}\int_{\R^n\setminus B_R}\frac{|w(x) - w(y)|^p}{|x-y|^{n+sp}}\dif x \dif y + \int_{B_R}\int_{B_R} \frac{|w(x) - w(y)|^p}{|x-y|^{n+sp}}\dif x \dif y \\
   &\leq2C_1 \left(\int_{B_R}\int_{\R^n\setminus B_R}\frac{|v(x) - v(y)|^p}{|x-y|^{n+sp}}\dif x \dif y + \int_{B_R}\int_{B_R} \frac{|v(x) - v(y)|^p}{|x-y|^{n+sp}}\dif x \dif y\right).
  \end{split}
 \end{equation}
We have
\begin{equation}
 v(x) - v(y) = (1-\eta(x))(u(x) - u(y)) - (\eta(x) - \eta(y))(u(y) - (u)_{B_{6\rho}}).
\end{equation}
Combining this with \eqref{eq:comparison-simplified}, we get
\begin{equation}\label{eq:v-formulaused}
 \begin{split}
 &\int_{B_r}\int_{\R^n} \frac{|u(x) - u(y)|^p}{|x-y|^{n+sp}}\dif x \dif y\\
 &\aleq_{p,\mathcal{N}} \int_{B_R}\int_{B_R}|1-\eta(x)|^p \frac{|u(x) - u(y)|^p}{|x-y|^{n+sp}} \dif x \dif y  + \int_{B_R}\int_{\R^n \setminus B_R} |1-\eta(x)|^p \frac{|u(x) - u(y)|^p}{|x-y|^{n+sp}} \dif x \dif y \\
 &\quad +\int_{B_R}\int_{\R^n} \frac{|\eta(x) - \eta(y)|^p|u(y)-(u)_{B_{6\rho}}|^p}{|x-y|^{n+sp}} \dif x \dif y.
 \end{split}
\end{equation}
We observe that for the last term of the right-hand side of \eqref{eq:v-formulaused} we have
\begin{equation}\label{eq:third-term1}
 \begin{split}
  \int_{B_R}\int_{\R^n} &\frac{|\eta(x) - \eta(y)|^p|u(y)-(u)_{B_{6\rho}}|^p}{|x-y|^{n+sp}} \dif x \dif y\\
  &\le \int_{B_R} \int_{\{x\in \R^n\colon 0\le|x-y|\le R-r\}} \frac{|\eta(x) - \eta(y)|^p|u(y)-(u)_{B_{6\rho}}|^p}{|x-y|^{n+sp}} \dif x \dif y\\
  &\quad + \int_{B_R} \int_{\{x\in \R^n\colon R-r \le|x-y|\le \infty\}} \frac{|\eta(x) - \eta(y)|^p|u(y)-(u)_{B_{6\rho}}|^p}{|x-y|^{n+sp}} \dif x \dif y.
  \end{split}
\end{equation}
For the first term on the right-hand side of (\ref{eq:third-term1}), we use $|\eta(x)-\eta(y)|\aleq \frac{|x-y|}{R-r}$ for all $x,y\in \R^n$. Then, by change of variables,
\begin{equation}\label{eq:third-term}
 \begin{split}
  &\int_{B_R}\int_{\R^n} \frac{|\eta(x) - \eta(y)|^p|u(y)-(u)_{B_{6\rho}}|^p}{|x-y|^{n+sp}} \dif x \dif y\\
  &\aleq_\eta \frac{1}{(R-r)^p} \int_{B_R}\int_{B_{R-r}}\frac{|u(y) - (u)_{B_{6\rho}}|^p}{|z|^{n+(s-1)p}}\dif z \dif y + \int_{B_{R}}|u(y) - (u)_{B_{6\rho}}|^p\int_{R-r}^{\infty} \nu^{-sp -1} \dif \nu \dif y\\
  &\aleq (R-r)^{-sp} \int_{B_{R}} |u(y) - (u)_{B_{6\rho}}|^p \dif y\\
  &\aleq (R-r)^{-sp} \int_{B_{2\rho}} |u(y) - (u)_{B_{6\rho}}|^p \dif y.
  \end{split}
\end{equation}
Since $\eta\equiv 1$ on $B_r$, we can estimate the first term of the right-hand side of \eqref{eq:v-formulaused} by
\begin{equation}\label{eq:first-term}
 \begin{split}
  \int_{B_R} \int_{B_R} |1-\eta(x)|^p \frac{|u(x) - u(y)|^p}{|x-y|^{n+sp}} \dif x \dif y
  &\le \int_{B_R} \int_{B_R\setminus B_r} \frac{|u(x) - u(y)|^p}{|x-y|^{n+sp}} \dif x \dif y\\
  &\le \int_{B_R\setminus B_r} \int_{\R^n} \frac{|u(x) - u(y)|^p}{|x-y|^{n+sp}} \dif x \dif y.
 \end{split}
\end{equation}
As $\eta\equiv 0$ on $\R^n\backslash B_R$, for the second term of the right-hand side of \eqref{eq:v-formulaused} we have
\begin{equation*}
\begin{split}
 \int_{B_R}&\int_{\R^n\setminus B_R} |1-\eta(x)|^p \frac{|u(x) - u(y)|^p}{|x-y|^{n+sp}} \dif x \dif y\\
 &=\int_{B_R}\int_{\R^n\setminus B_R} \frac{|u(x) - u(y)|^p}{|x-y|^{n+sp}} \dif x \dif y\\
 &= \int_{B_r}\int_{\R^n\setminus B_R} \frac{|u(x) - u(y)|^p}{|x-y|^{n+sp}} \dif x \dif y + \int_{B_R\setminus B_r}\int_{\R^n\setminus B_R} \frac{|u(x) - u(y)|^p}{|x-y|^{n+sp}} \dif x \dif y\\
 &= \int_{B_r}\int_{\R^n\setminus B_{R+2r}} \frac{|u(x) - u(y)|^p}{|x-y|^{n+sp}} \dif x \dif y + \int_{B_r}\int_{B_{R+2r}\setminus B_R} \frac{|u(x) - u(y)|^p}{|x-y|^{n+sp}} \dif x \dif y\\
 &\quad +\int_{B_R\setminus B_r}\int_{\R^n\backslash B_{R}}\frac{|u(x) - u(y)|^p}{|x-y|^{n+sp}} \dif x \dif y.
\end{split}
 \end{equation*}
 Thus, 
 \begin{equation}\label{eq:secondterm1}
\begin{split}
 &\int_{B_R}\int_{\R^n\setminus B_R} |1-\eta(x)|^p \frac{|u(x) - u(y)|^p}{|x-y|^{n+sp}} \dif x \dif y\\
&\aleq_p \int_{B_r}\int_{\R^n\setminus B_{R+2r}} \frac{|u(x) - (u)_{B_R\setminus B_r}|^p}{|x-y|^{n+sp}} \dif x \dif y + \int_{B_r}\int_{\R^n\setminus B_{R+2r}} \frac{|u(y) - (u)_{B_R\setminus B_r}|^p}{|x-y|^{n+sp}} \dif x \dif y\\
 &\quad +\int_{B_r}\int_{B_{R+2r}\setminus B_R} \frac{|u(x) - u(y)|^p}{|x-y|^{n+sp}} \dif x \dif y+\int_{B_R\setminus B_r}\int_{\R^n}\frac{|u(x) - u(y)|^p}{|x-y|^{n+sp}} \dif x \dif y.
\end{split}
 \end{equation}
We notice that if $z\in B_R \setminus B_r$, 
$y\in B_r$ and $x\in \R^n\setminus B_{R+2r}$, then $|x-y|\geq R+r$ and so,
\begin{align*}
   |x-z|\leq |x-y|+(R+r)\leq  2|x-y|.
\end{align*}
Thus, the first term of the right-hand side of \eqref{eq:secondterm1} can be estimated in the following way
\begin{equation*}
 \begin{split}
  \int_{B_r}\int_{\R^n\setminus B_{R+2r}}& \frac{|u(x) - (u)_{B_R\setminus B_r}|^p}{|x-y|^{n+sp}} \dif x \dif y \\
  &= \sum_{k=0}^\infty \int_{B_{r- \frac{2^k-1}{2^k}(R-r)}\setminus B_{r- \frac{2^{k+1}-1}{2^{k+1}}(R-r)}}\int_{\R^n\setminus B_{R+2r}} \frac{|u(x) - (u)_{B_R\setminus B_r}|^p}{|x-y|^{n+sp}} \dif x \dif y\\
  &\quad + \int_{B_r\setminus B_{2r-R}}\int_{\R^n\setminus B_{R+2r}} \frac{|u(x) - (u)_{B_R\setminus B_r}|^p}{|x-y|^{n+sp}} \dif x \dif y.
 \end{split}
\end{equation*}
Then, we have
\begin{equation}\label{eq:secondtermfirstterm_2a}
 \begin{split}
  &\int_{B_r}\int_{\R^n\setminus B_{R+2r}} \frac{|u(x) - (u)_{B_R\setminus B_r}|^p}{|x-y|^{n+sp}} \dif x \dif y \\
  &\aleq \sum_{k=0}^\infty\frac{1}{R^n - r^n} \int_{B_{r- \frac{2^k-1}{2^k}(R-r)}\setminus B_{r- \frac{2^{k+1}-1}{2^{k+1}}(R-r)}}\int_{\R^n\setminus B_{R+2r}} \int_{B_R\setminus B_r }\frac{|u(x) - u(z)|^p}{|x-y|^{n+sp}} \dif z \dif x \dif y\\
  &\quad + \frac{1}{R^n-r^n}\int_{B_r\setminus B_{2r-R}}\int_{\R^n\setminus B_{R+2r}}\int_{B_R\setminus B_r} \frac{|u(x) - u(z)|^p}{|x-y|^{n+sp}} \dif z \dif x \dif y\\
  &\aleq \sum_{k=0}^\infty\frac{1}{R^n - r^n} \int_{B_{r- \frac{2^k-1}{2^k}(R-r)}\setminus B_{r- \frac{2^{k+1}-1}{2^{k+1}}(R-r)}}\int_{\R^n\setminus B_{R+2r}} \int_{B_R\setminus B_r }\frac{|u(x) - u(z)|^p}{|x-z|^{n+sp}} \dif z\dif x \dif y\\
  &\quad + \frac{1}{R^n-r^n}\int_{B_r\setminus B_{2r-R}}\int_{\R^n\setminus B_{R+2r}}\int_{B_R\setminus B_r} \frac{|u(x) - u(z)|^p}{|x-z|^{n+sp}} \dif z \dif x \dif y\\
  &\aleq \sum_{k=1}^\infty\frac{(r- \frac{2^k-1}{2^k}(R-r))^n - (r- \frac{2^{k+1}-1}{2^{k+1}}(R-r))^n}{(R^n-r^n)}\int_{\R^n\setminus B_{R+2r}} \int_{B_R\setminus B_r }\frac{|u(x) - u(z)|^p}{|x-z|^{n+sp}} \dif z \dif x \\
  &\quad + \frac{r^n - (2r-R)^n}{(R^n-r^n)} \int_{\R^n\setminus B_{R+2r}} \int_{B_R\setminus B_r }\frac{|u(x) - u(z)|^p}{|x-z|^{n+sp}} \dif z \dif x.
 \end{split}
\end{equation}
Note that
\begin{align*}
    &\frac{(r- \frac{2^k-1}{2^k}(R-r))^n - (r- \frac{2^{k+1}-1}{2^{k+1}}(R-r))^n}{(R^n-r^n)}\\&= \frac{\frac{1}{2^{k+1}}(R-r)((r- \frac{2^k-1}{2^k}(R-r))^{n-1}+...+(r- \frac{2^{k+1}-1}{2^{k+1}}(R-r))^{n-1})}{(R-r)(R^{n-1}+rR^{n-2}+...+r^{n-2}R+r^{n-1})}\\
    &\leq \frac{\frac{1}{2^{k+1}} R^{n-1}}{r^{n-1}}\leq \frac{1}{2^{k+1}}\frac{(2\rho)^{n-1}}{\rho^{n-1}}\aleq \frac{1}{2^{k+1}}.
\end{align*}
Similarly
$$\frac{r^n - (2r-R)^n}{(R^n-r^n)}\aleq 1.$$
Using the above two inequalities in (\ref{eq:secondtermfirstterm_2a}) gives
\begin{equation}\label{eq:secondtermfirstterm}
 \begin{split}
  &\int_{B_r}\int_{\R^n\setminus B_{R+2r}} \frac{|u(x) - (u)_{B_R\setminus B_r}|^p}{|x-y|^{n+sp}} \dif x \dif y\\
  &\aleq \sum_{k=0}^\infty \frac{1}{2^{k+1}} \int_{\R^n\setminus B_{R+2r}} \int_{B_R\setminus B_r }\frac{|u(x) - u(z)|^p}{|x-z|^{n+sp}} \dif z \dif x + \int_{\R^n\setminus B_{R+2r}} \int_{B_R\setminus B_r }\frac{|u(x) - u(z)|^p}{|x-z|^{n+sp}} \dif z \dif x\\
  &\aleq \int_{\R^n\setminus B_{R+2r}} \int_{B_R\setminus B_r }\frac{|u(x) - u(z)|^p}{|x-z|^{n+sp}} \dif z \dif x \\
  &\leq \int_{B_R\setminus B_r }\int_{\R^n}\frac{|u(x) - u(z)|^p}{|x-z|^{n+sp}} \dif x \dif z.
 \end{split}
\end{equation}
As for the second term of the right-hand side of \eqref{eq:secondterm1} we have
\begin{equation}\label{eq:secondtermsecondterm}
\begin{split}
 \int_{B_r}&\int_{\R^n\setminus B_{R+2r}} \frac{|u(y) - (u)_{B_R\setminus B_r}|^p}{|x-y|^{n+sp}} \dif x \dif y\\
 &\aleq (R-r)^{-sp} \int_{B_r} |u(y) - (u)_{B_R\setminus B_r}|^p \dif y\\
 &\aleq (R-r)^{-sp} \brac{\int_{B_r} |u(y) - (u)_{B_{6\rho}}|^p \dif y + \int_{B_r}|(u)_{B_R\setminus B_r}-(u)_{B_{6\rho}}|^p \dif y}\\
 &\aleq (R-r)^{-sp} \int_{B_r} |u(y) - (u)_{B_{6\rho}}|^p \dif y + r^n(R-r)^{-n-sp}\int_{B_R\setminus B_r}|u(y)-(u)_{B_{6\rho}}|^p \dif y\\
 &\aleq (R-r)^{-sp} \int_{B_{2\rho}} |u(y) - (u)_{B_{6\rho}}|^p \dif y + \rho^n(R-r)^{-n-sp}\int_{B_{2\rho}}|u(y)-(u)_{B_{6\rho}}|^p \dif y.
\end{split}
 \end{equation}
 We now look at the third term of the right-hand side of \eqref{eq:secondterm1}. We have
 \begin{equation}\label{eq:thirdtermsecondterm}
\begin{split}
 \int_{B_r}&\int_{B_{R+2r}\setminus B_{R}} \frac{|u(x) - u(y)|^p}{|x-y|^{n+sp}} \dif x \dif y\\
 &\aleq_p\int_{B_r} \int_{B_{R+2r}\setminus B_{R}} \frac{|u(x) - (u)_{B_{6\rho}}|^p}{|x-y|^{n+sp}} \dif x \dif y+\int_{B_r}\int_{B_{R+2r}\setminus B_{R}} \frac{|u(y) - (u)_{B_{6\rho}}|^p}{|x-y|^{n+sp}} \dif x \dif y\\
 &\aleq (R-r)^{-sp} \brac{\int_{B_{R+2r\backslash B_R}} |u(x) - (u)_{B_{6\rho}}|^p \dif x + \int_{B_r} |u(y) - (u)_{B_{6\rho}}|^p \dif y }\\
 &\aleq (R-r)^{-sp} \int_{B_{6\rho}} |u(x) - (u)_{B_{6\rho}}|^p \dif x.
\end{split}
 \end{equation}
Therefore, combining \eqref{eq:secondterm1}, \eqref{eq:secondtermfirstterm}, \eqref{eq:secondtermsecondterm} and \eqref{eq:thirdtermsecondterm}, we obtain an estimate for the second term of the right-hand side of \eqref{eq:v-formulaused}
\begin{equation}\label{eq:secondterm}
\begin{split}
  \int_{B_R}\int_{\R^n\setminus B_R} &|1-\eta(x)|^p \frac{|u(x) - u(y)|^p}{|x-y|^{n+sp}} \dif x \dif y\\
  &\aleq \int_{B_R\setminus B_r }\int_{\R^n}\frac{|u(x) - u(z)|^p}{|x-z|^{n+sp}} \dif x \dif z +(R-r)^{-sp} \int_{B_{6\rho}} |u(y) - (u)_{B_{6\rho}}|^p \dif y\\
  &\quad + \rho^n(R-r)^{-n-sp}\int_{B_{6\rho}}|u(y)-(u)_{B_{6\rho}}|^p \dif y.
\end{split}
\end{equation}
Plugging in \eqref{eq:third-term}, \eqref{eq:first-term}, and \eqref{eq:secondterm} into \eqref{eq:v-formulaused}, we get for a constant $C>0$ independent of $u,R,r,$ and $ \rho$
\begin{equation}
 \begin{split}
\int_{B_r}\int_{\R^n}& \frac{|u(x) - u(y)|^p}{|x-y|^{n+sp}}\dif x \dif y\\
&\le C\int_{B_R\setminus B_r }\int_{\R^n} \frac{|u(x) - u(y)|^p}{|x-y|^{n+sp}} \dif x \dif y + C(R-r)^{-sp} \int_{B_{6\rho}} |u(y) - (u)_{B_{6\rho}}|^p \dif y\\
&\quad + C\rho^n(R-r)^{-n-sp}\int_{B_{6\rho}}|u(y)-(u)_{B_{6\rho}}|^p \dif y.
 \end{split}
\end{equation}
Thus, by the hole-filling trick we get for $0<\theta = \frac{C}{C+1}<1$ and for all $\rho\leq r<R\leq 2\rho$
\begin{equation}
\begin{split}
 \int_{B_r}\int_{\R^n}& \frac{|u(x) - u(y)|^p}{|x-y|^{n+sp}}\dif x \dif y\\
&\le \theta \int_{B_R}\int_{\R^n}\frac{|u(x) - u(y)|^p}{|x-y|^{n+sp}} \dif x \dif y + C(R-r)^{-sp} \int_{B_{6\rho}} |u(y) - (u)_{B_{6\rho}}|^p \dif y\\
&\quad + C\rho^n(R-r)^{-n-sp}\int_{B_{6\rho}}|u(y)-(u)_{B_{6\rho}}|^p \dif y.
\end{split}
\end{equation}
Applying the iteration lemma, \Cref{la:iteration}, we get for all $\rho\le r\le2\rho$
\begin{equation}
\begin{split}
\int_{B_r}\int_{\R^n}& \frac{|u(x) - u(y)|^p}{|x-y|^{n+sp}}\dif x \dif y\\
&\aleq (2\rho - r)^{-sp} \int_{B_{6\rho}} |u(y) - (u)_{B_{6\rho}}|^p \dif y + \rho^n(2\rho-r)^{-n-sp}\int_{B_{6\rho}}|u(y)-(u)_{B_{6\rho}}|^p \dif y.
\end{split}
\end{equation}
Taking $r=\rho$ we obtain
\begin{equation}
\int_{B_\rho}\int_{\R^n} \frac{|u(x) - u(y)|^p}{|x-y|^{n+sp}}\dif x \dif y \aleq \rho^{-sp} \int_{B_{6\rho}} |u(y) - (u)_{B_{6\rho}}|^p \dif y.
\end{equation}
\end{proof}
\section{Decay Estimate}\label{section decay}
The key step in proving the H\"older continuity of the minimizers is the decay estimate. The minimizer satisfy a decay in energy around a point where the minimizer has a small normalized energy. This can be used to prove the H\"older continuity around that point.
\begin{proposition}[Decay estimate]\label{th:decay}
Assume $1<p_1 < p_2 <\infty$ and  $0<s_1 < s_2<1$. There exists $\delta = \delta(s_1,s_2,p_1,p_2,n,N)> 0$ such that the following holds for any $s \in (s_1,s_2)$ and $p \in (p_1,p_2)$ such that $n-\delta<sp<n$:\\
Let $\mathcal{N}$ be as in \Cref{assumptions}. For any $\theta\in \left(0,\frac{1}{2}\right)$, there exist $\varepsilon >0$ such that if $u\in \dot{W}^{s,p}(\R^n,\mathcal{N})$ is a minimizing $W^{s,p}$-harmonic map in $B_1$ with
\begin{equation}\label{blow up 1}
    \int_{B_1}\int_{\R^n}\frac{|u(x)-u(y)|^p}{|x-y|^{n+sp}}dx\;dy\;<\varepsilon^p,
\end{equation}
then
\begin{equation}\label{eq:blowup-smallness}
    \theta^{sp-n}\int_{B_\theta}\int_{\R^n}\frac{|u(x)-u(y)|^p}{|x-y|^{n+sp}}dx\;dy\;\leq \frac{1}{2}\int_{B_1}\int_{\R^n}\frac{|u(x)-u(y)|^p}{|x-y|^{n+sp}}dx\;dy.
\end{equation}
\end{proposition}

\begin{proof}
    Assume the claim is false. Then for each $0<\theta<\frac{1}{2}$, there exist a sequence $\{\varepsilon_i\}_{i=1}^\infty \subset (0,1)$ such that $\varepsilon_i\to 0$ as $i\to \infty$, and a sequence $\{u_i\}_{i=1}^\infty$ of $W^{s,p}$-minimizing harmonic maps such that
    \begin{equation}\label{eq:Wsp norm of u_i}
        \int_{B_1}\int_{\R^n}\frac{|u_i(x)-u_i(y)|^p}{|x-y|^{n+sp}}dx\;dy\;=\varepsilon_i^p
    \end{equation}
    and
    \begin{equation}\label{eq:blow up 3}
    \theta^{sp-n}\int_{B_\theta}\int_{\R^n}\frac{|u_i(x)-u_i(y)|^p}{|x-y|^{n+sp}}dx\;dy\;> \frac{1}{2}\varepsilon_i^p.
\end{equation}
We define the normalized maps
$$v_i:=\frac{u_i-(u_i)_{B_1}}{\varepsilon_i}.$$
We have $(v_i)_{B_1}=0$ and
\begin{equation}\label{eq:bound for v}
    \int_{B_1}\int_{\R^n}\frac{|v_i(x)-v_i(y)|^p}{|x-y|^{n+sp}}dx\;dy\;=\frac{1}{\varepsilon_i^p}\int_{B_1}\int_{\R^n}\frac{|u_i(x)-u_i(y)|^p}{|x-y|^{n+sp}}dx\;dy=1.
\end{equation}
Also, by Jensen's inequality and (\ref{eq:Wsp norm of u_i})
\begin{align*}
    \int_{B_1}|v_{i}(x)|^p\;dx&= \frac{1}{\varepsilon_i^p}\int_{B_1}|u_i(x)-(u)_{B_1}|^p\;dx\\
    &\leq \frac{1}{\varepsilon_i^p}\int_{B_1}\barint_{B_1}|u_i(x)-u_i(y)|^p\;dx\;dy\\
&\leq C(n,s,p) \frac{1}{\varepsilon_i^p}\int_{B_1}\barint_{B_1}\frac{|u_i(x)-u_i(y)|^p}{|x-y|^{n+sp}}\;dx\;dy\\
    &\leq C(n,s,p).
\end{align*}
In particular,
\begin{equation*} 
    \Vert v_i\Vert_{W^{s,p}(B_1,\R^N)}\leq C(n,s,p).
\end{equation*}
By weak compactness and \Cref{th:Rellish} (by re-labeling up to a subsequence), we may assume 
$$v_i\to v\;\;\;\text{weakly in }W^{s,p}(B_1,\R^N),$$
$$v_i\to v\;\;\;\text{strongly in }L^p_{loc}(B_1,\R^N),$$
$$v_i\to v\;\;\;\text{p.w. a.e. in }B_1,\;\;\;\;\;$$
where $v\in W^{s,p}(B_1,\R^N)$ and
\begin{equation}\label{eq:Wsp bound for v}
    \Vert v\Vert_{W^{s,p}(B_1,\R^N)}\leq C(n,s,p).
\end{equation}
Let $\pi:B_\rho(\mathcal{N})\to \mathcal{N}$ be the nearest point projection. For $u\in \mathcal{N}$, let $\Pi(u)$ denote the orthogonal projection onto the tangent space $T_{u}\mathcal{N}$ and $\Pi^{\perp}(u)$ denotes $I-\Pi(u)$. Note that $\Pi(u)$ is a symmetric matrix given by
$$\pi_{kj}(u)=\partial_k \pi^j(\pi(u)), \;\;\;1\leq k,j\leq N$$
(see \cite[Lemma A.1]{kasha-schikorra-functional}). Take any $\varphi\in C^\infty_c(B_{1/2},\R^N)$, for $t>0$ small enough $u_i+t\varphi$ take values in $B_\rho(\mathcal{N})$. Since $u_i$ is the minimizer, we have
$$\frac{d}{dt}\Big|_{t=0}\int_{\R^n}\int_{\R^n}\frac{|(\pi\circ (u_i + t\varphi))(x)-(\pi\circ (u_i + t\varphi))(y)|^p}{|x-y|^{n+sp}}dx\;dy \;=0.$$
This gives
$$\int_{\R^n}\int_{\R^n}\frac{|u_i(x)-u_i(y)|^{p-2}(u_i(x)-u_i(y))(D\pi(u_i(x))\varphi(x)-D\pi (u_i(y))\varphi(y))}{|x-y|^{n+sp}}dx\;dy \;=0.$$
We rewrite this as
\begin{align*}
    \int_{\R^n}\int_{\R^n}&\frac{|u_i(x)-u_i(y)|^{p-2}(u_i(x)-u_i(y))(\varphi(x)-\varphi(y))}{|x-y|^{n+sp}}dx\;dy\\
     =& \int_{\R^n}\int_{\R^n}\frac{|u_i(x)-u_i(y)|^{p-2}(u_i(x)-u_i(y))(\Pi^\perp (u_i(x))\varphi(x)-\Pi^\perp (u_i(y))\varphi(y))}{|x-y|^{n+sp}}dx\;dy\\
    = &\int_{\R^n}\int_{\R^n}\frac{|u_i(x)-u_i(y)|^{p-2}(u_i(x)-u_i(y))(\Pi^\perp (u_i(x))-\Pi^\perp (u_i(y)))\varphi(x)}{|x-y|^{n+sp}}dx\;dy\\
    &+ \int_{\R^n}\int_{\R^n}\frac{|u_i(x)-u_i(y)|^{p-2}(u_i(x)-u_i(y))\Pi^\perp (u_i(y))(\varphi(x)-\varphi(y))}{|x-y|^{n+sp}}dx\;dy.
\end{align*}
Now, we estimate the two integrals of the above equation separately. Firstly, since $\Pi^\perp$ is uniformly Lipschitz on $\mathcal{N}$ we have
\begin{align*}
    \int_{\R^n}\int_{\R^n}&\frac{|u_i(x)-u_i(y)|^{p-2}(u_i(x)-u_i(y))(\Pi^\perp (u_i(x))-\Pi^\perp (u_i(y)))\varphi(x)}{|x-y|^{n+sp}}dx\;dy\\
    &\leq C(\mathcal{N})\int_{\R^n}\int_{B_{1/2}} \frac{|u_i(x)-u_i(y)|^p|\varphi(x)|}{|x-y|^{n+sp}}\;dx\;dy.
\end{align*}
For the second term, since $u_i(x),u_i(y)\in \mathcal{N}$ we have $u_i(x)=\pi(u_i(x)),\;u_i(y)=\pi(u_i(y))$. Then, by the Taylor expansion of $\pi$ \cite[Lemma A.1]{kasha-schikorra-functional},
\begin{equation*}
\begin{split}
    |(u_i(x)-u_i(y))-\Pi(u_i(y))&(u_i(x)-u_i(y))|\\&=|(\pi(u_i(x))-\pi(u_i(y)))-\Pi(u_i(y))(u_i(x)-u_i(y))|\\
    &\leq C(\mathcal{N}) |u_i(x)-u_i(y)|^2.
\end{split}
\end{equation*}
Thus,
\begin{align*}
    \bigg|\int_{\R^n}\int_{\R^n}&\frac{|u_i(x)-u_i(y)|^{p-2}(u_i(x)-u_i(y))\Pi^\perp (u_i(y))(\varphi(x)-\varphi(y))}{|x-y|^{n+sp}}dx\;dy \;\bigg|\\
    &\leq C(\mathcal{N})\int_{\R^n}\int_{B_{1/2}} \frac{|u_i(x)-u_i(y)|^p|\varphi(x)|}{|x-y|^{n+sp}}dx\;dy.
\end{align*}
Combining the above two bounds and rewriting $u_i$ in terms of $v_i$ give
\begin{equation}\label{eq:EL for v_i}
\begin{split}
    \bigg|\int_{\R^n}\int_{\R^n}&\frac{|v_i(x)-v_i(y)|^{p-2}(v_i(x)-v_i(y))(\varphi(x)-\varphi(y))}{|x-y|^{n+sp}}dx\;dy\;\bigg| \\&\hspace{1.2in}\leq C(\varphi,\mathcal{N}) \;\varepsilon_i\int_{\R^n}\int_{B_1} \frac{|v_i(x)-v_i(y)|^p}{|x-y|^{n+sp}}dx\;dy
    \\&\hspace{1.2in}= C(\varphi,\mathcal{N}) \;\varepsilon_i
\end{split}
\end{equation}
where the last equality follows from (\ref{eq:bound for v}).
Thus, we get
that the right hand side of (\ref{eq:EL for v_i}) converges to $0$ as $i\to \infty$.
Since supp$(\varphi)\subset B_{1/2}$, the left hand side of (\ref{eq:EL for v_i}) can be written as
\begin{align*}
    \int_{\R^n}\int_{\R^n}&\frac{|v_i(x)-v_i(y)|^{p-2}(v_i(x)-v_i(y))(\varphi(x)-\varphi(y))}{|x-y|^{n+sp}}dx\;dy\\
    =&\int_{B_1}\int_{B_1}\frac{|v_i(x)-v_i(y)|^{p-2}(v_i(x)-v_i(y))(\varphi(x)-\varphi(y))}{|x-y|^{n+sp}}dx\;dy\\
    &+2\int_{\R^n\backslash B_1}\int_{B_1}\frac{|v_i(x)-v_i(y)|^{p-2}(v_i(x)-v_i(y))(\varphi(x)-\varphi(y))}{|x-y|^{n+sp}}dx\;dy\\
    =:&\; \mathcal{A}_i+\mathcal{B}_i.
\end{align*}
We define
$$V_i(x,y):=\frac{|v_i(x)-v_i(y)|^{p-2}(v_i(x)-v_i(y))}{|x-y|^{s(p-1)}}.$$
Note that
\begin{align*}
    \int_{B_1}\int_{B_1}|V_i(x,y)|^{p'}\;\frac{dx\;dy}{|x-y|^{n}}&=\int_{B_1} \int_{B_1}\left|\frac{|v_i(x)-v_i(y)|^{p-2}(v_i(x)-v_i(y))}{|x-y|^{s(p-1)}}\right|^{p'}\frac{dx\;dy}{|x-y|^{n}}\\&\leq\int_{B_1}\int_{B_1}\frac{|v_i(x)-v_i(y)|^{p}}{|x-y|^{n+sp}}dx\;dy\leq C(n,s,p)
\end{align*}
where $\frac{1}{p}+\frac{1}{p'}=1$. Up to a further subsequence, by weak compactness, $V_i$ converges weakly in $L^{p'}(B_1\times B_1)$ with respect to the measure $\frac{dx\times dy}{|x-y|^n}$. But we know $V_i\to V$ pointwise a.e. in $B_1$ where $$V(x,y):=\frac{|v(x)-v(y)|^{p-2}(v(x)-v(y))}{|x-y|^{s(p-1)}}.$$ Thus,
$$\mathcal{A}_i\to \int_{B_1}\int_{B_1}\frac{|v(x)-v(y)|^{p-2}(v(x)-v(y))(\varphi(x)-\varphi(y))}{|x-y|^{n+sp}}dx\;dy.$$
Now, we have
$$\mathcal{B}_i= 2\int_{B_{1/2}}\int_{\R^n\backslash B_1}\frac{|v_i(x)-v_i(y)|^{p-2}(v_i(x)-v_i(y))\varphi(x)}{|x-y|^{n+sp}}dy\;dx.$$
Let $$f_i:= \chi_{B_{1/2}}\int_{\R^n\backslash B_1}\frac{|v_i(x)-v_i(y)|^{p-2}(v_i(x)-v_i(y))}{|x-y|^{n+sp}}dy.$$
Let $p'=\frac{p}{p-1}$. For $x\in B_{1/2}$,
\begin{align*}
    \mu_x(\R^n\backslash B_1)&:= \int_{\R^N\backslash B_1}\frac{1}{|x-y|^{n+sp}}dy\\
    &\leq C(n,s,p)\int_{1/2}^\infty \frac{\nu^{n-1}}{\nu^{n+sp}}d\nu\leq C(n,s,p).
\end{align*}
So, we can apply Jensen's inequality and equation (\ref{eq:bound for v}) to get
\begin{align*}
    \Vert f_i\Vert_{L^{p'}(\R^n,\R^N)}^{p'}&=\int_{B_{1/2}}\left|\int_{\R^n\backslash B_1}\frac{|v_i(x)-v_i(y)|^{p-2}(v_i(x)-v_i(y))}{|x-y|^{n+sp}} dy\right|^{p'} dx\\
    &=\int_{B_{1/2}}\mu_x(\R^n\backslash B_1)^{p'}\left|\barint_{\R^n\backslash B_1}|v_i(x)-v_i(y)|^{p-2}(v_i(x)-v_i(y))\; d\mu_x(y)\right|^{p'} dx\\
&\leq C(n,s,p)\int_{B_{1/2}}\int_{\R^n\backslash B_1}\frac{|v_i(x)-v_i(y)|^{p}}{|x-y|^{n+sp}} dy\; dx\leq C(n,s,p).
\end{align*}
\noindent Thus, relabeling up to a subsequence we get $f_i\to f$ in $L^{p'}(\R^n,\R^N)$, as $i\to \infty$, for some $f\in L^{p'}(\R^n,\R^N)$ with $\Vert f\Vert_{L^{p'}(\R^n,\R^N)}\leq C(n,s,p)$.
Also, we have
$$\lim_{i\to \infty}\int_{B_{1/2}}\int_{\R^n\backslash B_1}\frac{|v_i(x)-v_i(y)|^{p-2}(v_i(x)-v_i(y))\varphi(x)}{|x-y|^{n+sp}}dy\;dx=\int_{\R^n}f\;\varphi.$$
Letting $i\to \infty$ in (\ref{eq:EL for v_i}) gives
\begin{equation}\label{eq:fraclapalcian in B1 for v}
    \int_{B_1}\int_{B_1}\frac{|v(x)-v(y)|^{p-2}(v(x)-v(y))(\varphi(x)-\varphi(y))}{|x-y|^{n+sp}}dx\;dy=-2\int_{\R^n} f\;\varphi
\end{equation}
for any $\varphi\in C^\infty_{c}(B_{1/2},\R^N)$. \\
Take a smooth cut-ff function $\eta\geq 0$ such that $\eta\equiv 1$ on $B_{1/2}$ and $\eta\equiv 0$ outside $B_{3/4}$. Define $\tilde{v}:=\eta v$. Let $$T(t):=|t|^{p-2}t$$ for $t\in \R^N$. For $t\neq 0$, we get $$(D T_j)_i=(p-2)|t|^{p-4}t_it_j+|t|^{p-4}\delta_{ij}.$$
$$|DT(t)|\leq C(p-1)|t|^{p-2}$$
since $p\geq 2$. Therefore, for $x,y\in B_1$,
\begin{align*}
    &|T(\eta(x)v(x)-\eta(y)v(y))-T(\eta(x)(v(x)-v(y)))|\\
    &\leq \int_{0}^1 |DT(\eta(x)(v(x)-v(y))+s\;v(y)(\eta(x)-\eta(y)))|\;|v(y)(\eta(x)-\eta(y))|\;ds\\
    &\aleq \left(\eta(x)^{p-2}|v(x)-v(y)|^{p-2}+ |v(y)|^{p-2}|\eta(x)-\eta(y)|^{p-2}\right)|v(y)||\eta(x)-\eta(y)|.
\end{align*}
Thus, for any $\varphi\in C_c^\infty(B_{1/4},\R^N)$, we have
\begin{align*}
    &(T(\eta(x)v(x)-\eta(y)v(y))- T(\eta(x)(v(x)-v(y))))(\varphi(x)-\varphi(y))\\
    &\aleq \left(\eta(x)^{p-2}|v(x)-v(y)|^{p-2}+ |v(y)|^{p-2}|\eta(x)-\eta(y)|^{p-2}\right)|v(y)||\eta(x)-\eta(y)||\varphi(x)-\varphi(y)|.
\end{align*}
So,
\begin{align*}
    \RN{1}&:=\int_{B_1}\int_{B_1} \frac{|\tilde{v}(x)-\tilde{v}(y)|^{p-2}(\tilde{v}(x)-\tilde{v}(y))(\varphi(x)-\varphi(y))}{|x-y|^{n+sp}}dx \;dy\\
    &=\int_{B_1}\int_{B_1} \frac{|\eta(x)v(x)-\eta(y)v(y)|^{p-2} (\eta(x)v(x)-\eta(y)v(y))(\varphi(x)-\varphi(y))}{|x-y|^{n+sp}}dx\;dy\\
    &= \int_{B_1}\int_{B_1}\frac{T(\eta(x)v(x)-\eta(y)v(y))(\varphi(x)-\varphi(y))}{|x-y|^{n+sp}}dx\;dy\\
    &\aleq \RN{1}_a+\RN{1}_b+\RN{1}_c
\end{align*}
where
$$\RN{1}_a:= \int_{B_1}\int_{B_1} \frac{(\eta(x))^{p-2}|v(x)-v(y)|^{p-2}\eta(x)(v(x)-v(y))(\varphi(x)-\varphi(y))}{|x-y|^{n+sp}}dx\;dy,$$
$$\RN{1}_b:= \int_{B_1}\int_{B_1} \frac{(\eta(x))^{p-2}|v(x)-v(y)|^{p-2}|v(y)||\eta(x)-\eta(y)||\varphi(x)-\varphi(y)|}{|x-y|^{n+sp}}dx\;dy,$$
$$\RN{1}_c:= \int_{B_1}\int_{B_1} \frac{|v(y)|^{p-2}|\eta(x)-\eta(y)|^{p-2}|v(y)||\eta(x)-\eta(y)||\varphi(x)-\varphi(y)|}{|x-y|^{n+sp}}dx\;dy.$$
Observe that
\begin{align*}
    \RN{1}_a=& \int_{B_1}\int_{B_1} \frac{|v(x)-v(y)|^{p-2}(\eta(x))^{p-1}(v(x)-v(y))(\varphi(x)-\varphi(y))}{|x-y|^{n+sp}}dx\;dy\\
    = &\int_{B_1}\int_{B_1} \frac{|v(x)-v(y)|^{p-2}(v(x)-v(y))((\eta(x))^{p-1}\varphi(x)-(\eta(y))^{p-1}\varphi(y))}{|x-y|^{n+sp}}dx\;dy\\
    &- \int_{B_1}\int_{B_1} \frac{|v(x)-v(y)|^{p-2}(v(x)-v(y))\varphi(y)((\eta(x))^{p-1}-(\eta(y))^{p-1})}{|x-y|^{n+sp}}dx\;dy\\
    =:&\; \RN{1}_{aa}-\RN{1}_{ab}.
\end{align*}

By (\ref{eq:fraclapalcian in B1 for v}), we get
\begin{align*}
    \RN{1}_{aa}&=\int_{B_1}\int_{B_1} \frac{|v(x)-v(y)|^{p-2}(v(x)-v(y))((\eta(x))^{p-1}\varphi(x)-(\eta(y))^{p-1}\varphi(y))}{|x-y|^{n+sp}}dx\;dy\\&=-2\int_{\R^n}f\;(\eta^{p-1}\varphi).
\end{align*}
$$$$
This implies
\begin{equation*}
    |\RN{1}_{aa}|\leq 2\Vert f \Vert_{L^{p'}(\R^n,\R^N)}\Vert \eta^{p-1}\varphi \Vert_{L^{p}(\R^n,\R^N) } \leq C(\eta,n,s,p) \Vert \varphi \Vert_{L^{p}(B_{1/4},\R^N) }
\end{equation*}
since $\Vert f\Vert_{L^{p'}(\R^n,\R^N)}\leq C(n,s,p).$

Define $$h_1:=\chi_{B_1}\int_{B_1} \frac{|v(x)-v(y)|^{p-2}(v(x)-v(y))((\eta(x))^{p-1}-(\eta(y))^{p-1})}{|x-y|^{n+sp}}dx.$$
Let $p'=\frac{p}{p-1}$.
For $y\in B_{1}$,
\begin{align*}
    \mu_y( B_1)&:= \int_{B_1}\frac{1}{|x-y|^{n+s-1}}dx\\
    &\leq C(n,s)\int_{0}^2 \frac{\nu^{n-1}}{\nu^{n+s-1}}d\nu\leq C(n,s).
\end{align*}
So, we can apply Jensen's inequality and equation (\ref{eq:bound for v}) to get
\begin{align*}
    \Vert h_1\Vert_{L^{p'}(\R^n,\R^N)}^{p'}&=\int_{B_{1}}\left|\int_{ B_1}\frac{|v(x)-v(y)|^{p-2}(v(x)-v(y))((\eta(x))^{p-1}-(\eta(y))^{p-1})}{|x-y|^{n+sp}} dx\right|^{p'} dy\\
    &\leq \Vert \nabla\eta\Vert_{L^\infty(B_1)}\int_{B_{1}}\left(\int_{ B_1}\frac{|v(x)-v(y)|^{p-1}|x-y|}{|x-y|^{s(p-1)}|x-y|^{n+s}} dx\right)^{p'} dy\\
    &\leq C(\eta)\int_{B_{1}}\mu_y( B_1)^{p'}\left(\barint_{ B_1}\frac{|v(x)-v(y)|^{p-1}}{|x-y|^{s(p-1)}}\; d\mu_y(x)\right)^{p'} dy\\
&\leq C(n,s,p,\eta)\int_{B_{1}}\int_{B_1}\frac{|v(x)-v(y)|^{p}}{|x-y|^{n+sp}} dx\; dy\leq C(n,s,p,\eta).
\end{align*}
Then,
$$\Vert h_1 \Vert_{L^{p'}(\R^n,\R^N)}^{p'}\leq C(n,s, p,\eta )<\infty$$
and

$$\RN{1}_{ab}=\int_{B_1}\int_{B_1} \frac{|v(x)-v(y)|^{p-2}(v(x)-v(y))((\eta(x))^{p-1}-(\eta(y))^{p-1})}{|x-y|^{n+sp}}\varphi(y)\;dx\;dy=\int_{\R^n}h_1\varphi.$$
This gives
\begin{equation}
    |\RN{1}_{ab}|\leq \Vert h_1 \Vert_{L^{p'}(\R^n,\R^N)}\Vert \varphi \Vert_{L^{p}(B_{1/4},\R^N) } \leq C(\eta,n,s,p) \Vert \varphi \Vert_{L^{p}(B_{1/4},\R^N) }.
\end{equation}

Since $\eta\equiv 1$ on $B_{1/2}$ and $\varphi\in C_c^\infty(B_{1/4},\R^N)$, we have
\begin{align*}
    \RN{1}_b=& \int_{B_1}\int_{B_1} \frac{(\eta(x))^{p-2}|v(x)-v(y)|^{p-2}|v(y)||\eta(x)-\eta(y)||\varphi(x)-\varphi(y)|}{|x-y|^{n+sp}}dx\;dy\\
    =&\int_{B_{1/4}}\int_{B_1\backslash B_{1/2}} \frac{(\eta(x))^{p-2}|v(x)-v(y)|^{p-2}|v(y)||\eta(x)-\eta(y)||\varphi(y)|}{|x-y|^{n+sp}}dx\;dy\\
    &+ \int_{B_1\backslash B_{1/2}} \int_{B_{1/4}} \frac{(\eta(x))^{p-2}|v(x)-v(y)|^{p-2}|v(y)||\eta(x)-\eta(y)||\varphi(x)|}{|x-y|^{n+sp}}dx\;dy\\
    =:&\; \RN{1}_{ba}+\RN{1}_{bb}.
\end{align*}

Note that in the above integrals $|x-y|\geq \frac{1}{4}$. Therefore, we get
\begin{align*}
\RN{1}_{ba}&\aleq_{\eta,p}\int_{B_{1/4}}\int_{B_1\backslash B_{1/2}}|v(x)|^{p-2}|v(y)| |\varphi(y)|dx\;dy+\int_{B_{1/4}}\int_{B_1\backslash B_{1/2}}|v(y)|^{p-1} |\varphi(y)|dx\;dy\\
&\aleq \Vert v\Vert_{L^{p-2}(B_1,\R^N)}^{p-2}\Vert v\Vert_{L^{p'}(B_{1/4},\R^N)}\Vert \varphi\Vert_{L^{p}(B_{1/4},\R^N)}+\Vert v\Vert_{L^{p}(B_{1/4},\R^N)}^{p-1}\Vert \varphi\Vert_{L^{p}(B_{1/4},\R^N)}\\
&\leq C(\eta,n,s,p)\Vert \varphi\Vert_{L^{p}(B_{1/4},\R^N)}
\end{align*}
and
\begin{align*}
    \RN{1}_{bb}&\aleq_{\eta,p}\int_{B_1\backslash B_{1/2}}\int_{B_{1/4}}|v(x)|^{p-2}|v(y)| |\varphi(x)|dx\;dy+\int_{B_1\backslash B_{1/2}}\int_{B_{1/4}}|v(y)|^{p-1} |\varphi(x)|dx\;dy\\
    &\aleq \Vert v\Vert_{L^{1}(B_1,\R^N)}\Vert v\Vert_{L^{\left(\frac{p-2}{p-1}\right)p}(B_{1/4},\R^N)}^{p-2}\Vert \varphi\Vert_{L^{p}(B_{1/4},\R^N)}+\Vert v\Vert_{L^{p-1}(B_{1},\R^N)}^{p-1}\Vert \varphi\Vert_{L^{p}(B_{1/4},\R^N)}\\
    &\leq C(\eta,n,s,p)\Vert \varphi\Vert_{L^{p}(B_{1/4},\R^N)}.
\end{align*}

We also have
\begin{align*}
    \RN{1}_c:=& \int_{B_1}\int_{B_1} \frac{|v(y)|^{p-1}|\eta(x)-\eta(y)|^{p-1}|\varphi(x)-\varphi(y)|}{|x-y|^{n+sp}}dx\;dy\\
    =&\int_{B_{1/4}}\int_{B_1\backslash B_{1/2}} \frac{|v(y)|^{p-1}|\eta(x)-\eta(y)|^{p-1}|\varphi(y)|}{|x-y|^{n+sp}}dx\;dy\\
    &+ \int_{B_1\backslash B_{1/2}} \int_{B_{1/4}} \frac{|v(y)|^{p-1}|\eta(x)-\eta(y)|^{p-1}|\varphi(x)|}{|x-y|^{n+sp}}dx\;dy\\
    =:&\; \RN{1}_{ca}+\RN{1}_{cb}.
\end{align*}

Note that in the above integrals $|x-y|\geq \frac{1}{4}$. Therefore, we get
\begin{align*}
    \RN{1}_{ca}&\aleq_{\eta,p}\int_{B_{1/4}}\int_{B_1\backslash B_{1/2}}|v(y)|^{p-1}|\varphi(y)|dx\;dy\\
    &\aleq \Vert v\Vert_{L^{p}(B_1,\R^N)}^{p-1}\Vert \varphi\Vert_{L^{p}(B_{1/4},\R^N)}\\
    &\leq C(\eta,n,s,p)\Vert \varphi\Vert_{L^{p}(B_{1/4},\R^N)}
\end{align*}
and
\begin{align*}
    \RN{1}_{cb}&\aleq_{\eta,p}\int_{B_{1/4}}\int_{B_1\backslash B_{1/2}}|v(y)|^{p-1}|\varphi(x)|dx\;dy\\
    &\aleq \Vert v\Vert_{L^{p}(B_{1/4},\R^N)}^{p-1}\Vert \varphi\Vert_{L^{p}(B_{1/4},\R^N)}\\
    &\leq C(\eta,n,s,p)\Vert \varphi\Vert_{L^{p}(B_{1/4},\R^N)}.
\end{align*}

Combining the bounds on $\RN{1}_{aa},\RN{1}_{ab},\RN{1}_{ba},\RN{1}_{bb},\RN{1}_{ca},\RN{1}_{cb}$, we get
\begin{equation}\label{eq:bound for I}
    |\RN{1}|\leq C(\eta,n,s,p)\Vert \varphi\Vert_{L^{p}(B_{1/4},\R^N)}.
\end{equation}

Also, for any $\varphi\in C_c^\infty(B_{1/4},\R^N)$, we have
\begin{align*}
    \RN{2}&:=\int_{\R^n\backslash B_1}\int_{B_1} \frac{|\tilde{v}(x)-\tilde{v}(y)|^{p-2}(\tilde{v}(x)-\tilde{v}(y))(\varphi(x)-\varphi(y))}{|x-y|^{n+sp}}dx \;dy\\
    &=\int_{\R^n\backslash B_1}\int_{B_1} \frac{|\eta(x)v(x)-\eta(y)v(y)|^{p-2} (\eta(x)v(x)-\eta(y)v(y))(\varphi(x)-\varphi(y))}{|x-y|^{n+sp}}dx\;dy\\
    &= \int_{\R^n\backslash B_1}\int_{B_{1/4}} \frac{|v(x)|^{p-2} v(x)\varphi(x)}{|x-y|^{n+sp}}dx\;dy
\end{align*}
since $\eta\equiv 0$, $\varphi\equiv 0$ on $\R^n\backslash B_1$ and $\eta\equiv 1$ on $B_{1/4}$. Now, for any $x\in B_{1/4}$, we have
\begin{align*}
    \int_{\R^n\backslash B_1}\frac{1}{|x-y|^{n+sp}}dy&= \int_{\R^n\backslash B_1(-x)}\frac{1}{|y|^{n+sp}}dy\\
    &\leq \int_{\R^n\backslash B_{1/4}}\frac{1}{|y|^{n+sp}}dy\\
    &\aleq \int_{1/4}^\infty \frac{1}{r^{n+sp}}r^{n-1}dr= \left.\frac{-1}{sp}\;\frac{1}{r^{sp}}\right|_{1/4}^\infty<C(n,s,p).
\end{align*}
Thus,
$$\RN{2}\aleq \Vert v\Vert_{L^{p}(B_{1/4},\R^N)}^{p-1}\Vert \varphi\Vert_{L^{p}(B_{1/4},\R^N)}<\infty.$$
That is,
\begin{equation}\label{eq:bound for II}
    |\RN{2}|\leq C(\eta,n,s,p)\Vert \varphi\Vert_{L^{p}(B_{1/4},\R^N)}.
\end{equation}
For any $\varphi\in C^\infty_c(B_{1/4},\R^N)$, define $$L[\varphi]:= \int_{\R^n}\int_{\R^n} \frac{|\tilde{v}(x)-\tilde{v}(y)|^{p-2}(\tilde{v}(x)-\tilde{v}(y))(\varphi(x)-\varphi(y))}{|x-y|^{n+sp}}dx \;dy.$$
Note that $$L[\varphi]=\RN{1}+2\;\RN{2}.$$ From (\ref{eq:bound for I}) and (\ref{eq:bound for II}), we get $L$ is a bounded linear functional on $L^p(B_{1/4},\R^N)$. Hence, there exists a $h\in L^{p'}(B_{1/4},\R^N)$ such that $\Vert h\Vert_{L^{p'}(B_{1/4},\R^N)}\leq C(\eta,n,s,p)$ and
$$\int_{\R^n}\int_{\R^n} \frac{|\tilde{v}(x)-\tilde{v}(y)|^{p-2}(\tilde{v}(x)-\tilde{v}(y))(\varphi(x)-\varphi(y))}{|x-y|^{n+sp}}dx \;dy=\int_{\R^n}h\;\varphi.$$

So, we get $u $ is a weak solution of
$$(-\Delta_p)^s \tilde{v}=h\;\;\;\;\text{in}\;\;B_{1/4},$$
where $h\in L^{p'}(B_{1/4},\R^N).$ Since supp($\tilde{v}$)$\subset B_1$ and $v\in L^p(B_1,\R^N)$, we get $\tilde{v}\in L^{p-1}_{sp}(\R^n,\R^N)$. Also,
\begin{align*}
    \int_{B_1}\int_{B_1}&\frac{|\tilde{v}(x)-\tilde{v}(y)|^p}{|x-y|^{n+sp}}\;dx\;dy\\&=\int_{B_1}\int_{B_1}\frac{|\eta(x)v(x)-\eta(y)v(y)|^p}{|x-y|^{n+sp}}\;dx\;dy\\&\aleq \int_{B_1}\int_{B_1}\frac{|\eta(x)(v(x)-v(y))|^p}{|x-y|^{n+sp}}\;dx\;dy+\int_{B_1}\int_{B_1}\frac{|v(y)(\eta(x)-\eta(y))|^p}{|x-y|^{n+sp}}\;dx\;dy\\
    &\aleq \int_{B_1}\int_{B_1}\frac{|v(x)-v(y)|^p}{|x-y|^{n+sp}}\;dx\;dy+\Vert \nabla \eta\Vert_{L^\infty(B_1)}\int_{B_1}\int_{B_1}\frac{|x-y|^p}{|x-y|^{n+sp}}\;dx\;dy\\
    &\aleq_{\eta} \int_{B_1}\int_{B_1}\frac{|v(x)-v(y)|^p}{|x-y|^{n+sp}}\;dx\;dy+\int_{B_1}|v(y)|^p\int_{B_1(-y)}|x|^{-n+(1-s)p}\;dx\;dy\\
    &\aleq \int_{B_1}\int_{B_1}\frac{|v(x)-v(y)|^p}{|x-y|^{n+sp}}\;dx\;dy+\int_{B_1}|v(y)|^p\int_{B_2(0)}|x|^{-n+(1-s)p}\;dx\;dy\\
    &\aleq_{n} \int_{B_1}\int_{B_1}\frac{|v(x)-v(y)|^p}{|x-y|^{n+sp}}\;dx\;dy+\int_{B_1}|v(y)|^p\int_{0}^2t^{-n+(1-s)p}t^{n-1}\;dt\;dy\\
    &\aleq \Vert v\Vert_{W^{s,p}(B_1,\R^N)}^p\leq C(n,s,p)<\infty.
\end{align*}
So, $\tilde{v}\in W^{s,p}(B_1,\R^N)$. Also, since $\eta\equiv 0$ outside $B_{3/4}$, we have
\begin{align*}
    \int_{B_1}\int_{\R^n\backslash B_1}&\frac{|\tilde{v}(x)-\tilde{v}(y)|^p}{|x-y|^{n+sp}}\;dx\;dy\\&=\int_{B_{3/4}}\int_{\R^n\backslash B_1}\frac{|\eta(x)v(x)|^p}{|x-y|^{n+sp}}\;dx\;dy
    \\&\aleq \int_{B_{3/4}}|v(x)|^p\int_{\R^n\backslash B_1}\frac{1}{|x-y|^{n+sp}}\;dx\;dy \\
    &\aleq \int_{B_{3/4}}|v(x)|^p \int_{1/4}^\infty \frac{t^{n-1}}{t^{n+sp}} dt\;dy\\
    &\leq \Vert v\Vert_{W^{s,p}(B_1,\R^N)}^p\leq C(n,s,p)<\infty.
\end{align*}
By \Cref{th:alternate for BLS}, under our assumptions, there exists a constant $C=C(n,s,p,\delta)$ such that 
\begin{align*}
    [\tilde{v}]_{C^{0,\alpha}(B_{r_0},\R^N)}\leq \frac{C}{r_0^{\frac{\delta}{p}}}&\left( \int_{B_{4r_0}}\int_{\R^n}\frac{|\tilde{v}(x)-\tilde{v}(y)|^p}{|x-y|^{n+sp}}dx\;dy+\Vert h \Vert_{L^{p'}(B_{2r_0},\R^N)}[\tilde{v}]_{W^{s,p}(B_{2r_0},\R^N)} \right)^{1/p}
\end{align*}
where $\alpha:=(sp+\delta-n )/p$ and $r_0<\frac{1}{8}$. This implies for all $x,y \in B_{r_0}$, we have
\begin{equation}\label{eq:holder bound for u}
    |\tilde{v}(x)-\tilde{v}(y)|\leq C(n,s,p,\eta)|x-y|^\alpha.
\end{equation}
But $\tilde{v}=v$ on $B_{1/2}$. Hence, for any $0<r<{r_0}$, we get
\begin{equation}\label{eq:holder cont of v}
    r^{-n}\int_{B_r} \barint_{B_r} |v(x)-v(y)|^p\;dx\;dy\leq C_1 r^{\alpha p}
\end{equation}
where $C_1=C(n,s,p,\eta)$.

Since $v_i\to v$ in $L^p_{\text{loc}}(B_1,\R^N)$, for $\theta<\frac{r_0}{6}$ we can find an $i(\theta)$, such that for all $i\geq i(\theta) $
\begin{equation}\label{eq:strong convergence of v_i}
    \theta^{-n}\int_{B_{6\theta}}|v_i-v|^p\leq \theta.
\end{equation}
By \Cref{th:caccioppoli}, for all $i\geq i(\theta)$
\begin{align*}
    \theta^{sp-n}\int_{B_\theta}\int_{\R^n}&\frac{|u_i(x)-u_i(y)|^p}{|x-y|^{n+sp}}dx\;dy\leq C\theta^{-n}\int_{B_{6\theta}} |u_i-(u_i)_{B_{6\theta}}|^p\\
    &= C \theta^{-n}\varepsilon_i^p \int_{B_{6\theta}}|v_i-(v_i)_{B_{6\theta}}|^p \\
    &\leq C\varepsilon_i^p \theta^{-n}\left( 2\int_{B_{6\theta}}|v_i-v|^p\;+\int_{B_{6\theta}} \barint_{B_{6\theta}}|v(x)-v(y)|^p \;dx\;dy \right)\\
    &\leq C' \varepsilon^p_i (\theta +\theta^ {\alpha p })
\end{align*}
where the last inequality follows from (\ref{eq:strong convergence of v_i}) and (\ref{eq:holder cont of v}) and $C'=C(n,s,p,\eta,N,\mathcal{N})$. Choose $\theta$ small enough so that $C'(\theta+\theta^{\alpha p})<\frac{1}{2}$ to get a contradiction to (\ref{eq:blow up 3}).
\end{proof}


\section{Proof of the Main Theorem}\label{section proof}

Now, we present the proof of our main theorem. To simplify, we split our main theorem into two results.  Firstly, we prove that the minimizers are locally H\"older continuous around the points where the minimizers satisfy the decay estimate. This follows from Campanato's theorem [\Cref{th:campanato}]. Secondly, we discuss the dimension of the singular set. The points, around which the minimizers fail to satisfy a decay in energy, belong to a singular set.
\begin{theorem}[H\"{o}lder continuity]\label{th:nloc-holdercont}
 Let $\delta,s,p,\theta$ and $\varepsilon$ be as in \Cref{th:decay}. Let $\Omega$ and $\mathcal{N}$ be as in \Cref{assumptions}. Assume that $u\in \dot{W}^{s,p}(\R^n,\mathcal{N})$ is a minimizing $W^{s,p}$-harmonic map in $\Omega$. If $R>0$ and $z_0\in \Omega$ such that $B_R(z_0)\subset \Omega$ and
 \[
  R^{sp-n}\int_{B_R(z_0)} \int_{\R^n}\frac{|u(x) - u(y)|^p}{|x-y|^{n+sp}} \dif x \dif y \le \eps^p,
 \]
then $u\in C^{0,\alpha}(B_{\frac{R}{2}}(z_0), \mathcal{N})$ for some $\alpha>0$.
\end{theorem}
\begin{proof}
Define $u_{z_0,R}(x):= u(Rx+z_0)$.
Then $u_{z_0,R}\in \dot{W}^{s,p}(\R^n,\mathcal{N})$ and is a minimizing $W^{s,p}$-harmonic map in $B_1$. We apply \Cref{th:decay} for $u_{z_0,R}$, which satisfy the smallness assumption \eqref{eq:blowup-smallness} and obtain
 \[
  (\theta R)^{sp-n}\int_{B_{\theta R}(z_0)}\int_{\R^n} \frac{|u(x) - u(y)|^p}{|x-y|^{n+sp}} \dif x \dif y< \frac12 R^{sp-n}\int_{B_R(z_0)}\int_{\R^n} \frac{|u(x) - u(y)|^p}{|x-y|^{n+sp}} \dif x \dif y.
 \]
We repeat this by applying theorem \Cref{th:decay} to $u_{x_0,\theta R} , \ldots, u_{x_0,\theta^j R}$ for each $k=0,\ldots, j$. We obtain
\begin{equation}\label{eq:holder-iteration}
\begin{split}
    (\theta^j R)^{sp-n} \int_{B_{\theta^j R}(z_0)} \int_{\R^n}  \frac{|u(x) - u(y)|^p}{|x-y|^{n+sp}} \dif x \dif y< \brac{\frac{1}{2}}^j R^{sp-n}\int_{B_R(z_0)} \int_{\R^n}  \frac{|u(x) - u(y)|^p}{|x-y|^{n+sp}} \dif x \dif y.
\end{split}
\end{equation}
For each $r\in(0,R)$, we choose a $j$ such that $\theta^{j+1} R \le r \le \theta ^j R$. Note that
$$\theta^j=\left(\frac{1}{2}\right)^{\log_{\frac{1}{2}} (\theta^j)}=\left(\frac{1}{2}\right)^{j\frac{\ln(\theta)}{\ln(1/2)}}=\left(\frac{1}{2}\right)^{\frac{j}{\beta}}$$
where $\beta:=\frac{\ln(1/2)}{\ln(\theta)}$. Since $\theta\in (0,1/2)$, note that $0<\beta<1$.
Now, we have
\begin{align*}
    r^{sp-n} \int_{B_r(z_0)} \int_{\R^n} \frac{|u(x) - u(y)|^p}{|x-y|^{n+sp}}\dif x \dif y &
  \le (\theta^{j+1} R)^{sp-n} \int_{B_{\theta^j R}(z_0)}\int_{\R^n}\frac{|u(x) - u(y)|^p}{|x-y|^{n+sp}}\dif x \dif y\\
  &=\theta^{sp-n}(\theta^{j} R)^{sp-n} \int_{B_{\theta^j R}(z_0)}\int_{\R^n}\frac{|u(x) - u(y)|^p}{|x-y|^{n+sp}}\dif x \dif y
\end{align*}
Thus, from \eqref{eq:holder-iteration} we get
\[
 \begin{split}
  r^{sp-n} \int_{B_r(z_0)} \int_{\R^n} &\frac{|u(x) - u(y)|^p}{|x-y|^{n+sp}}\dif x \dif y \\
  &\le \theta^{sp-n}\brac{\frac12}^j R^{sp-n}\int_{B_R(z_0)}\int_{\R^n}\frac{|u(x) - u(y)|^p}{|x-y|^{n+sp}}\dif x \dif y\\
  &= \theta^{sp -n} (\theta^j)^\beta R^{sp-n}\int_{B_R(z_0)}\int_{\R^n}\frac{|u(x) - u(y)|^p}{|x-y|^{n+sp}}\dif x \dif y\\
  &\le \theta^{sp-n-\beta} \left(\frac{r}{R}\right)^{\beta} R^{sp-n}\int_{B_R(z_0)}\int_{\R^n}\frac{|u(x) - u(y)|^p}{|x-y|^{n+sp}}\dif x \dif y.
 \end{split}
\]
Now, for any $z\in B_{\frac{R}{2}}(z_0)$ and $0<r<\frac{R}{2}$, we get
\begin{equation*}
    \begin{split}
        r^{-\beta}r^{sp-n}\int_{B_r(z)}\int_{\R^n}&\frac{|u(x) - u(y)|^p}{|x-y|^{n+sp}}\dif x \dif y\\&\aleq_{\theta} R^{-\beta}R^{sp-n}\int_{B_R(z_0)}\int_{\R^n}\frac{|u(x) - u(y)|^p}{|x-y|^{n+sp}}\dif x \dif y\\&\leq C(\theta).
    \end{split}
\end{equation*}
For all $z\in B_{\frac{R}{2}}(z_0)$ and $0<r<\frac{R}{2}$,
\[
\begin{split}
  \int_{B_r(z)}|u(x)-(u)_{B_r(z)}|^p \dif x
 &\le \int_{B_r(z)} \barint_{B_r(z)} |u(x) - u(y)|^p \dif x \dif y \\
 & \le  C r^{-n}r^{sp+n} \int_{B_r(z)} \int_{B_r(z)} \frac{|u(x) - u(y)|^p}{|x-y|^{n+sp}} \dif x \dif y\\
 & \le  C r^{sp} \int_{B_r(z)} \int_{\R^n} \frac{|u(x) - u(y)|^p}{|x-y|^{n+sp}} \dif x \dif y.
\end{split}
 \]

Here, the constant is invariant under scaling and translation of $B_r(z)$. Therefore,
\[
\begin{split}
 \sup_{\substack{r\in \left(0,\frac{R}{2}\right)\\ z\in B_{\frac{R}{2}}(z_0)}}r^{-n-\beta} &\int_{B_r(z)}|u(x)-(u)_{B_r(z)}|^p \dif x \\
 &\aleq_\theta  \sup_{\substack{r\in \left(0,\frac{R}{2}\right)\\ z\in B_{\frac{R}{2}}(z_0)}} r^{sp-n-\beta } \int_{B_r(z)} \int_{\R^n} \frac{|u(x) - u(y)|^p}{|x-y|^{n+sp}} \dif x \dif y\\
 &\leq C(\theta).
\end{split}
 \]
Thus, by \Cref{th:campanato} we obtain $u\in C^{0,\alpha}(B_{\frac{R}{2}}(x_0),\mathcal{N})$ where $\alpha:=\frac \beta p$ and $0<\alpha\leq 1$.
\end{proof}
\begin{theorem}[Dimension of the singular set]\label{corollaryregularity}
Let $\delta,s$ and $p$ be as in \Cref{th:decay}. Let $\Omega$ and $\mathcal{N}$ be as in \Cref{assumptions}. Then there exists $\zeta=\zeta(n,p,s)>0$ such that the following holds: if $u\in \dot{W}^{s,p}(\R^n,\mathcal{N})$ is a minimizing $W^{s,p}$-harmonic map in $\Omega$, then there exists a singular set $\Sigma\subset \Omega$(relatively closed) such that
$$\mathcal{H}^{n-sp-\zeta}(\Sigma)=0$$
and for any $x_0\in \Omega\backslash \Sigma$, $u$ is H\"older continuous in a small neighborhood of $x_0$.
\end{theorem}
\begin{proof}
Fix $\theta\in (0,1/2)$. Let $\varepsilon$ be as in Proposition \ref{th:decay}. Let $\varepsilon_1=\varepsilon^p$. Define
$$\Sigma:=\{x\in \Omega:\;u\notin C^{0,\alpha}(B_{\rho}(x),\mathcal{N})\;\text{for any }\rho>0\}.$$
It is clear that $\Sigma$ is relatively closed. Define
$$E:=\left\{x_0\in \Omega:\;\;\limsup_{R\to 0}R^{sp-n}\int_{B_R(x_0)}\int_{\R^n} \frac{|u(x)-u(y)|^p}{|x-y|^{n+sp}}dx\; dy >\varepsilon_1 \right\}.$$
If $x_0\notin E$, then by Theorem \ref{th:nloc-holdercont} we obtain $u$ is H\"older continuous around $x_0$. So we have
$$\Omega\backslash E\subset \Omega\backslash \Sigma.$$
This implies $\Sigma \subset E$.\\
It remains to estimate the size of $E$. From Lemma \ref{Frostmann}, we get $\mathcal{H}^{n-sp}(E)=0$. To get a better estimate we use the minimizing property once more, in form of the Caccioppoli-type estimate and Gehring's lemma.

Choose a number $p_2\in\left(1,\infty\right)$ such that $\max\{p,\frac{n}{n-s}\}<p_2<\frac{np}{n-sp}$. Define $q:=\frac{np_2}{n+sp_2}$. Then $q\in (1,p)$ and
$$s-\frac{n}{q}=-\frac{n}{p_2}.$$
Clearly, $p\in (1,p_2)$. Let $\Lambda>1$ be as in \Cref{Lemma C.4}. Let $x_0\in \Omega$. Choose $\rho_0>0$ such that $B_{6\Lambda\rho_0}(x_0)\subset \Omega$. Then, for $\rho<\rho_0$ we have
\begin{equation}\label{eq:lemma 2.9}
    \Vert u-(u)_{B_{6\rho}(x_0)}\Vert_{L^{p_2}(B_{6\rho}(x_0),\R^N)} \leq C_1\left(\int_{B_{6\Lambda\rho}(x_0 )}\left(\int_{B_{6\Lambda\rho}(x_0)}\frac{|u(x)-u(y)|^p}{|x-y|^{n+sp}}dx\right)^{\frac{q}{p}}dy\right)^{\frac{1}{q}}.
\end{equation}
    where the constant $C_1$ depends only on $s,p,p_2,q$ and the dimension.

Since $p_2>p$, we have
\begin{equation}\label{eq:L_p_2 bnd}
    \Vert u-(u)_{B_{6\rho}(x_0)}\Vert_{L^{p}(B_{6\rho}(x_0),\R^N)} \aleq_n \rho^{\frac{n}{p}-\frac{n}{p_2}}\Vert u-(u)_{B_{6\rho}(x_0)}\Vert_{L^{p_2}(B_{6\rho}(x_0),\R^N)}.
\end{equation}
From \Cref{th:caccioppoli}, we have
\[
\left(\int_{B_\rho(x_0)}\int_{\R^n}\frac{|u(x) - u(y)|^p}{|x-y|^{n+sp}}\dif x \dif y\right)^{1/p} \le C \rho^{-s }\left(\int_{B_{6\rho}(x_0)}|u(x)-(u)_{B_{6\rho}(x_0)}|^p dx\right)^{1/p}.
\]
Using \eqref{eq:lemma 2.9} and \eqref{eq:L_p_2 bnd}, we get
\begin{equation*}
    \begin{split}
        \Bigg(\int_{B_\rho(x_0)}\int_{\R^n}&\frac{|u(x) - u(y)|^p}{|x-y|^{n+sp}}\dif x \dif y\Bigg)^{1/p}\\& \le C\rho^{-s }\rho^{\frac{n}{p}-\frac{n}{p_2}}\left(\int_{B_{6\Lambda \rho}(x_0 )}\left(\int_{B_{6\Lambda \rho}(x_0)}\frac{|u(x)-u(y)|^p}{|x-y|^{n+sp}}dx\right)^{\frac{q}{p}}dy\right)^{\frac{1}{q}}.
    \end{split}
\end{equation*}
Since $s=\frac{n}{q}-\frac{n}{p_2}$, we have
\begin{equation*}
    \begin{split}
        \Bigg(\int_{B_\rho(x_0)}\int_{\R^n}&\frac{|u(x) - u(y)|^p}{|x-y|^{n+sp}}\dif x \dif y\Bigg)^{1/p} \\&\le C\rho^{\frac{n}{p}-\frac{n}{q}}\left(\int_{B_{6\Lambda \rho}(x_0 )}\left(\int_{B_{6\Lambda \rho}(x_0)}\frac{|u(x)-u(y)|^p}{|x-y|^{n+sp}}dx\right)^{\frac{q}{p}}dy\right)^{\frac{1}{q}}.
    \end{split}
\end{equation*}
This implies
\[
\left(\barint_{B_\rho(x_0)}\int_{\R^n}\frac{|u(x) - u(y)|^p}{|x-y|^{n+sp}}\dif x \dif y\right)^{1/p} \le C\left(\barint_{B_{6\Lambda \rho}(x_0 )}\left(\int_{B_{6\Lambda \rho}(x_0)}\frac{|u(x)-u(y)|^p}{|x-y|^{n+sp}}dx\right)^{\frac{q}{p}}dy\right)^{\frac{1}{q}}.
\]
Set
\[
 \Gamma(y) := \brac{\int_{\R^n}\frac{|u(x) - u(y)|^p}{|x-y|^{n+sp}}\dif x}^{\frac{1}{p}}.
\]
Then we have shown
\[
 \brac{\mvint_{B_\rho(x_0)} |\Gamma(y)|^p}^{\frac{1}{p}} \aleq \brac{\mvint_{B_{6\Lambda \rho}(x_0)} |\Gamma(y)|^q}^{\frac{1}{q}}.
\]
Since $q<p$, we can apply \Cref{th:gehring} and, we find that for some $\bar{p} > p$ we have

\[
 \brac{\mvint_{B_\rho(x_0)} |\Gamma(y)|^{\bar{p}}}^{\frac{1}{\bar{p}}} \aleq \brac{\mvint_{B_{6\Lambda\rho}(x_0)} |\Gamma(y)|^p}^{\frac{1}{p}}.
\]

In particular, for any compact set $K \subset \Omega$, we have
\[
 \int_{K} \brac{\int_{\R^n}  \frac{|u(x)-u(y)|^{p}}{|x-y|^{n+sp}} dx}^{\frac{\bar{p}}{p}} dy < \infty.
\]
We define
$$F:=\left\{x\in\Omega:\;\;\limsup_{r\to 0}r^{s\bar{p}-n}\int_{B_r(x)}\left(\int_{\R^n}\frac{|u(x)-u(y)|^p}{|x-y|^{n+sp}}dx\right)^{\frac{\bar{p}}{p}}\; dy>0\right\}.$$
By \Cref{Frostmann}, we get $\mathcal{H}^{n-s\bar{p}}(F)=0$.

Now, we have by Jensen's inequality
\begin{align*}
    \rho^{sp-n}\int_{B_\rho(x_0)}\int_{\R^n} &\frac{|u(x)-u(y)|^{p}}{|x-y|^{n+sp}} dx \; dy\\&\aleq_n \rho^{sp-n} \rho^{n-\frac{np}{\bar{p}}} \left(\int_{B_\rho(x_0)} \brac{\int_{\R^n}  \frac{|u(x)-u(y)|^{p}}{|x-y|^{n+sp}} dx}^{\frac{\bar{p}}{p}} dy\right)^{\frac{p}{\bar{p}}}\\
    &\aleq_n \rho^{sp-\frac{np}{\bar{p}}} \left(\int_{B_\rho(x_0)} \brac{\int_{\R^n}  \frac{|u(x)-u(y)|^{p}}{|x-y|^{n+sp}} dx}^{\frac{\bar{p}}{p}} dy\right)^{\frac{p}{\bar{p}}}\\
     &\aleq_n  \left(\rho^{s\bar{p}-n}\int_{B_\rho(x_0)} \brac{\int_{\R^n}  \frac{|u(x)-u(y)|^{p}}{|x-y|^{n+sp}} dx}^{\frac{\bar{p}}{p}} dy\right)^{\frac{p}{\bar{p}}}.
\end{align*}
This implies $E\subset F$. Hence, $\mathcal{H}^{n-s\bar{p}}(E)=0$.

Since $\Sigma\subset E$, we get $\mathcal{H}^{n-s\bar{p}}(\Sigma)=0$. Define $\zeta:=s\bar{p}-sp>0$, then $\mathcal{H}^{n-sp-\zeta}(\Sigma)=0$. We can conclude.

\end{proof}
\appendix
\section{H\"older regularity for fractional $p$-Laplace systems when $sp \approx n$: Proof of Theorem~\ref{th:alternate for BLS}}\label{s:nongeomplapsys}

Denote by $B_r(x_0)$ a ball of radius $r$ and center $x_0$, such that $B_{10r}(x_0) \subset  \Omega$.

By assumption we have for any $\varphi \in C_c^\infty(B_{4r}(x_0),\R^N)$
\begin{equation}\label{eq:pde}
 \int_{\R^n} \int_{\R^n} \frac{|u(x)-u(y)|^{p-2} (u(x)-u(y)) (\varphi(x)-\varphi(y))}{|x-y|^{n+sp}}\, dx dy = \int_{\R^n} f \varphi.
\end{equation}
For simplicity, assume $x_0=0$. Denote the annulus $A = B_{2r} \setminus B_r$, $\eta \in C_c^\infty(B_{2r},[0,1])$, $\eta \equiv 1$ in $B_r$, $|\nabla \eta| \aleq r^{-1}$. Set
\[
 \varphi(x) := \eta(x) (u(x)-(u)_{A}).
\]
Using the product rule
\[
 \varphi(x) - \varphi(y)  = \eta(x) (u(x)-u(y)) + (\eta(x)-\eta(y)) (u(y)-(u)_A),
\]
we have
\[
\begin{split}
 \int_{B_r} \int_{\R^n}& \frac{|u(x)-u(y)|^{p}}{|x-y|^{n+sp}}\, dx dy\\
 \leq &\int_{\R^n} \int_{\R^n} \eta(y) \frac{|u(x)-u(y)|^{p}}{|x-y|^{n+sp}}\, dx dy\\
 =&\int_{\R^n} \int_{\R^n} \frac{|u(x)-u(y)|^{p-2} (u(x)-u(y)) \eta(y) (u(x)-u(y))}{|x-y|^{n+sp}}\, dx dy\\
 \overset{\eqref{eq:pde}}{\leq} &\int_{\R^n} \int_{\R^n}  \frac{|u(x)-u(y)|^{p-1} |u(x)-(u)_A| |\eta(x)-\eta(y)|}{|x-y|^{n+sp}}\, dx dy\\
 &+\int_{\R^n} |f(x)| \eta(x) |u(x)-(u)_A|
 \end{split}
\]
\[
\begin{split}
 \;\;\;\;\;\;\;\;\;\leq &\int_{B_{4r} \setminus B_r} \int_{B_{4r}}  \frac{|u(x)-u(y)|^{p-1} |u(x)-(u)_A| |\eta(x)-\eta(y)|}{|x-y|^{n+sp}}\, dx dy\\
 &+\int_{B_{4r}} \int_{B_{4r} \setminus B_{r}}  \frac{|u(x)-u(y)|^{p-1} |u(x)-(u)_A| |\eta(x)-\eta(y)|}{|x-y|^{n+sp}}\, dx dy \\
 &+\int_{\R^n \setminus B_{4r}} \int_{B_{4r}}  \frac{|u(x)-u(y)|^{p-1} |u(x)-(u)_A| }{|x-y|^{n+sp}}\, dx dy\\
 &+\int_{B_{4r}} \int_{\R^n \setminus B_{4r}}  \frac{|u(x)-u(y)|^{p-1} |u(x)-(u)_A| |\eta(y)|}{|x-y|^{n+sp}}\, dx dy\\
 &+\|f\|_{L^{\frac{p}{p-1}}(\Omega,\R^N)} \brac{\int_{B_{2r}} |u(x)-(u)_A|^p}^{\frac{1}{p}}
  \end{split}
\]
\[
\begin{split}
  \leq &\eps \int_{B_{4r}} \int_{\R^n} \frac{|u(x)-u(y)|^p}{|x-y|^{n+sp}}\, dx\, dy  \\
 & +C(\eps) \int_{B_{4r} \setminus B_{r}} \int_{B_{4r}}  \frac{ |u(x)-(u)_A|^p |\eta(x)-\eta(y)|^p}{|x-y|^{n+sp}}\, dx dy\\
 &+C(\eps) \int_{B_{4r}} \int_{B_{4r} \setminus B_{r}}  \frac{|u(x)-(u)_A|^p |\eta(x)-\eta(y)|^p}{|x-y|^{n+sp}}\, dx dy \\
 &+C(\eps) \int_{\R^n \setminus B_{4r}} (1+|y|)^{-n-sp} dy \int_{B_{4r}} |u(x)-(u)_A|^p dx \\
 &+C(\eps) r^n \int_{\R^n \setminus B_{4r}} (1+|x|)^{-n-sp} |u(x)-(u)_A|^p dx \\
 &+\|f\|_{L^{\frac{p}{p-1}}(\Omega,\R^N)} \brac{\int_{B_{2r}} |u(x)-(u)_A|^p}^{\frac{1}{p}}
 \end{split}
\]
\[
\begin{split}
 \int_{B_r} \int_{\R^n}& \frac{|u(x)-u(y)|^{p}}{|x-y|^{n+sp}}\, dx dy\\\aleq &\eps \int_{B_{4r}} \int_{\R^n} \frac{|u(x)-u(y)|^p}{|x-y|^{n+sp}}\, dx\, dy  \\
 &+C(\eps)  r^{-sp} \int_{B_{4r}} |u(x)-(u)_A|^p dx \\
 &+C(\eps) r^n \int_{\R^n \setminus B_{4r}} (1+|x|)^{-n-sp} |u(x)-(u)_A|^p dx \\
 &+r^s \|f\|_{L^{\frac{p}{p-1}}(\Omega,\R^N)} [u]_{W^{s,p}(\Omega,\R^N)}.
 \end{split}
\]

Observe that
\[
 C(\eps)  r^{-sp} \int_{B_{4r}} |u(x)-(u)_A|^p dx  \aleq \int_{B_{2r} \setminus B_{r}} \int_{\R^n} \frac{|u(x)-u(y)|^p}{|x-y|^{n+sp}} dx dy
\]
and
\[
\begin{split}
 r^n \int_{\R^n \setminus B_{4r}} &(1+|x|)^{-n-sp} |u(x)-(u)_A|^p dx \\
 \aleq&\int_{\R^n \setminus B_{4r}} \int_{A} \frac{|u(x)-u(y)|^p}{1+|x|^{n+sp}} dx dy\\
 \aleq&\int_{\R^n \setminus B_{4r}} \int_{B_{2r} \setminus B_{r}} \frac{|u(x)-u(y)|^p}{|x-y|^{n+sp}} dx dy\\
 =&\int_{B_{2r} \setminus B_{r}} \int_{\R^n} \frac{|u(x)-u(y)|^p}{|x-y|^{n+sp}} dx dy.
 \end{split}
\]

Consequently, we have
\[
\begin{split}
 \int_{B_{r}} \int_{\R^n} \frac{|u(x)-u(y)|^{p}}{|x-y|^{n+sp}}\, dx dy \leq& \eps \int_{B_{4r}} \int_{\R^n} \frac{|u(x)-u(y)|^{p}}{|x-y|^{n+sp}}\, dx dy\\
 &+ C(\eps)\int_{B_{2r} \setminus B_{r}} \int_{\R^n} \frac{|u(x)-u(y)|^{p}}{|x-y|^{n+sp}}\, dx dy\\
 &+C\, r^s \|f\|_{L^{\frac{p}{p-1}}(B_{2r},\R^N)} [u]_{W^{s,p}(B_{2r},\R^N)}.
 \end{split}
\]
Using the hole-filling argument we have found that for all $B_{10r}\subset \Omega$ we have
\[
\begin{split}
 \int_{B_r} &\int_{\R^n} \frac{|u(x)-u(y)|^{p}}{|x-y|^{n+sp}}\, dx dy \\
 \leq &\frac{\eps+C(\eps)} {1+C(\eps)} \int_{B_{4r}} \int_{\R^n} \frac{|u(x)-u(y)|^{p}}{|x-y|^{n+sp}}\, dx dy  + Cr^s \|f\|_{L^{\frac{p}{p-1}}(B_{2r},\R^N)} [u]_{W^{s,p}(B_{2r},\R^N)}.
 \end{split}
 \]
It is important to note that the constant $C$ is essentially independent of $s$ and $p$ in the sense that for any $\delta_0 > 0$, $\delta \in (0,\delta_0)$ and $s \in (s_0-\delta,s_0+\delta)$ $p \in (p_0-\delta,p_0+\delta)$ we can make the constant dependent only on $\delta_0$ and $s_0,p_0$.

Fix a ball $B_{4R}(x_0) \subset \Omega$. The above implies that for some $\tau \in (0,1)$ (depending only on $s_0,p_0,\delta_0$) and any (not necessarily concentric) ball $B_r(y) \subset B_R(x_0)$ we have
\[
\begin{split}
 \int_{B_r(y)} &\int_{\R^n} \frac{|u(x)-u(y)|^{p}}{|x-y|^{n+sp}}\, dx dy \\
 \leq &\tau \int_{B_{4R}(x_0)} \int_{\R^n} \frac{|u(x)-u(y)|^{p}}{|x-y|^{n+sp}}\, dx dy  + Cr^s \|f\|_{L^{\frac{p}{p-1}}(B_{2R}(x_0),\R^N)} [u]_{W^{s,p}(B_{2R}(x_0),\R^N)}.
 \end{split}
 \]
By iteration, there is a constant $\delta \in (0,s)$ only depending on $C$ (and thus $\delta_0$ and $s_0,p_0$) we find
\[
\begin{split}
    &\int_{B_r(y)} \int_{\R^n} \frac{|u(x)-u(y)|^{p}}{|x-y|^{n+sp}}\, dx dy \\  &\aleq \brac{\frac{r}{R}}^\delta\, \int_{B_{4R}(x_0)} \int_{\R^n} \frac{|u(x)-u(y)|^{p}}{|x-y|^{n+sp}}\, dx dy  + \brac{\frac{r}{R}}^\delta \|f\|_{L^{\frac{p}{p-1}}(B_{2R}(x_0),\R^N)} [u]_{W^{s,p}(B_{2R}(x_0),\R^N)}.
\end{split}
\]
Thus,
\[
\begin{split}
  \sup_{B_r(y) \subset B_R(x_0)}& r^{-sp-\delta} \int_{B_r(y)} |u(x)-(u)_{B_r(y)}|^p  \\
  \aleq &R^{-\delta} \brac{\int_{B_{4R}(x_0)} \int_{\R^n} \frac{|u(x)-u(y)|^{p}}{|x-y|^{n+sp}}\, dx dy  + \|f\|_{L^{\frac{p}{p-1}}(B_{2R}(x_0),\R^N)} [u]_{W^{s,p}(B_{2R}(x_0),\R^N)}}.
  \end{split}
\]
If $n<sp+\delta\leq n+p $, then by the theory of Campanato spaces (\Cref{th:campanato}) we find
\[
\begin{split}
    [u]&_{C^{0,\alpha}(B_{R}(x_0),\R^N)} \\&\aleq R^{-\frac{\delta}{p}} \brac{\int_{B_{4R}(x_0)} \int_{\R^n} \frac{|u(x)-u(y)|^{p}}{|x-y|^{n+sp}}\, dx dy  + \|f\|_{L^{\frac{p}{p-1}}(B_{2R}(x_0),\R^N)} [u]_{W^{s,p}(B_{2R}(x_0),\R^N)}}^{\frac{1}{p}}.
\end{split}
\]
where $\alpha=(sp+\delta-n)/p$. This concludes the proof of \Cref{th:alternate for BLS}.

\bibliographystyle{abbrv}
\bibliography{bib}

\end{document}